\documentclass[reqno]{amsart}

\usepackage{graphicx}
\usepackage{todonotes}
\usepackage[utf8]{inputenc}
\usepackage{amsfonts}
\usepackage{amsmath}
\usepackage{enumerate}
\usepackage{amsthm}
\usepackage{amssymb}
\usepackage[backend=bibtex]{biblatex}
\usepackage{graphicx}
\usepackage{subcaption}
\usepackage{tikz,pgfplots}
\usepackage{epsfig}
\usepackage{url}
\usepackage{bm}
\usepackage{float}
\usepackage{cases}
\usepackage{enumitem} 
\usepackage{centernot}
\usepackage{booktabs}
\usepackage{mathtools}

\newcommand{\R}{\mathbb{R}}

\newcommand{\p}{\partial}

\usepackage[colorlinks=true,linkcolor=black,citecolor=black,filecolor=black,urlcolor=black,menucolor=black]{hyperref}
\newtheorem{theorem}{Theorem}
\newtheorem{lemma}{Lemma}
\newtheorem{proposition}{Proposition}

\newtheorem{remark}{Remark}

\begin{document}

\title[Change of bifurcation type in FBP of a moving cell]{Change of bifurcation type in 2D free boundary model of a moving cell with nonlinear diffusion
}
\author{Leonid Berlyand}
\address{Department of Mathematics\\
         The Pennsylvania State University\\
         University Park, Pennsylvania 16802\\
         USA}
\email{lvb2@psu.edu}

\author[Oleksii Krupchytskyi]{Oleksii Krupchytskyi\textsuperscript{*}}
\address{Department of Mathematics\\
         The Pennsylvania State University\\
         University Park, Pennsylvania 16802\\
         USA}
\email{omk5165@psu.edu}
\thanks{$^*$Corresponding author. \textit{E-mail address}: \texttt{omk5165@psu.edu}}

\author{Tim Laux}
\address{Institute for Mathematics\\
         Heidelberg University\\
         Im Neuenheimer Feld 205\\
         D-69120 Heidelberg\\
         Germany}
\email{tim.laux@math.uni-heidelberg.de}

\maketitle

\begin{abstract}
 We introduce a 2D free boundary problem with nonlinear diffusion that models a living cell moving on a substrate.
 We prove that this nonlinearity results in a qualitative change of solution behavior compared to the linear diffusion case (Rybalko et al.\ TAMS 2023), namely the switch between supercritical and subcritical pitchfork bifurcation.

    Our objectives are twofold: 
    (i) develop a rigorous framework to prove existence of bifurcation and determine its type (subcritical vs. supercritical) and (ii) the derivation of explicit analytical formulas that control the change of bifurcation type in terms of physical parameters and explain the underlying biophysical mechanisms.

    While the standard way of applying the Crandall-Rabinowitz theorem via the solution operator seems difficult in our quasilinear PDE system, we apply the theorem directly, by developing a multidimensional, vectorial framework.
    To determine the bifurcation type, we extract the curvature of the bifurcating curve from the expansion of the solutions around the steady state.
    The formula for the curvature is obtained via a solvability condition where instead of the Fredholm alternative, we propose a test function trick, suited for free boundary problems.
    
    Our rigorous analytical results are in agreement with numerical observations from the physical literature in 1D (Drozdowski et al.\ Comm.\ Phys.\ 2023) and provide the first extension of this phenomenon to a 2D free boundary model.
    	\medskip
    

  \medskip

\noindent \textbf{Mathematical Subject Classification (MSC20)}: 35B32; 70K50; 92C17.
\end{abstract}
\setcounter{tocdepth}{2}
\tableofcontents
\section{Introduction}
\label{intro}

\subsection{Motivation and context}

Keller-Segel-type PDE systems in a domain with moving and deformable boundary arise in the modeling of motility (self-sustained motion) of living cells. Such motility is a hallmark of active matter (also known as active materials) which is a fast-growing field in both physics and applied mathematics \cite{aranson2022bacterial, gompper20202020, marchetti2013hydrodynamics, ron2023polarization}.  
From the mathematical perspective, there are two main PDE approaches in cell motility: via phase-field or free-boundary models.
Phase fields have been extensively used to study the evolution of the cell shape both analytically and numerically \cite{ berlyand2016phase, ziebert2016computational}. 
However, fundamental mathematical questions such as existence of traveling wave solutions, their emergence via bifurcations from a stationary solution, and stability can be better answered in the context of free boundary models.
\medskip

In this work we introduce and study a 2D model of cell motility with \emph{nonlinear} myosin diffusion which mathematically amounts to a coupled system of elliptic and parabolic PDEs in a free boundary setting with a nonlocal boundary condition. Our goals are twofold: (i) a rigorous proof of the existence of bifurcation and establishing its type (subcritical vs. supercritical) and (ii) derivation of analytical formulas that control the change of the bifurcation type in terms of physical parameters and explain the biophysical mechanisms underlying the bifurcation change.
The two different bifurcation types lead to crucially different scenarios of the onset of cell motion and are naturally connected to different stability behavior. 
In particular, subcritical bifurcation typically leads to bistability of the steady and motile states.
The present work is motivated by numerical studies of a 1D model in \cite{DroZieSch23}.
Our results confirm these findings and extend them to 2D, which provides connections to experimental studies on the onset of cell motility, e.g.~\cite{lacayo2007emergence, keren2008mechanism}.
Our findings are in stark contrast to the case of \emph{linear} myosin diffusion, in which only supercritical pitchfork bifurcation is observed~\cite{RybBer23}. 
In the special case of a fixed cell boundary and in the vanishing friction limit in the 1D model, formulas for the bifurcation change were derived via formal asymptotics in \cite{chelly2022cell}. 
The existence of the bifurcation for this special case can be established via simplified 1D analogs of the techniques proposed in the present work. 
The change of bifurcation type appears to be ubiquitous in active matter -- not just in cell motility. 
For example, recent experimental studies \cite{barnhart2011adhesion, barnhart2017adhesion} 
suggest that both supercritical and subcritical pitchfork bifurcations can appear, capturing different physics.
We believe that the analytical techniques developed in this paper will lead to a more general understanding of bifurcation phenomena across various problems of active matter.

\medskip

We briefly comment here on the literature on free boundary models. PDE problems in domains with moving and deformable boundaries arise in mathematical modeling in physics, materials science, and biophysics. {They date back to the seminal works on Stefan~\cite{Crank1984}} 
and Hele-Shaw~\cite{gustafsson2004conformal} problems.
In recent decades, free boundary models have been used to model tumor growth~\cite{friedman2006asymptotic, siewe2023cancer, ZhaBeiZha2023bif}, tissue growth~\cite{king2021free, alert2019active, berlyand2024bifurcation}, and cell motility, see e.g.~\cite{RecTruThi13,RecPutTru15, PutRecTru18},\cite{MogCop2020}, \cite{blanch2013spontaneous}, \cite{cucchi2022self, cucchi2020cahn}, \cite{alazard2022traveling}, \cite{BerRybSaf2022,  RybBer23, BerSafTru26}. 

\medskip

 Several mathematical papers address the existence of pitchfork bifurcation to traveling waves in cell motility, see e.g.~\cite{MeuAla2025, MeuMagGar2025, cucchi2022self}, \cite{RybBer23, BerRybSaf2022}, \cite{alazard2022traveling}.
 The proofs in these works are based on the Crandall-Rabinowitz theorem \cite{crandall1971bifurcation} in the functional setting based on the solution operator. This strategy was also applied to the analysis of tumor growth and plaque formation free boundary models, e.g.~\cite{ZhaBeiZha2023bif},\cite{ borisovich2005symmetry}.
 However, due to the nonlinear diffusion, our PDE is quasilinear (rather than the previously studied linear case). This makes establishing the existence of a {general} solution operator rather difficult, and instead we apply the Crandall-Rabinowitz theorem directly to the PDE system, which leads to verifying the transversality and simple eigenvalue condition in a multidimensional, vectorial setting.
The recent works~\cite{carrillo2023noise,carrillo2025well} rigorously establish the  bifurcation (where the noise level plays the role of the bifurcation parameter) from a homogeneous state to various patterns in a mean field PDE model for grid cells in the brains of mammals, as well as nonlinear stability of solutions. {The existence of traveling wave solutions in the presence of nonlinear diffusion was established for a fixed boundary domain in~\cite{du2025precise,audrito2019bistable}, whereas in the present paper we consider domains with moving and deformable boundary.}

\medskip
{The main mathematical novelty of our work is the rigorous derivation of an explicit formula that determines the bifurcation type in our 2D free boundary problem in terms of physical parameters (see Theorem~\ref{th:K2_f-la}).}
We expand the branch of traveling wave solutions around the steady state and find the curvature of the bifurcating curve at the bifurcation point in the third-order expansion. Instead of the Fredholm alternative, which easily applies in the absence of a free boundary, we introduce a suitable test function to extract this information. 
With this formula, we can prove the change of bifurcation type for relevant physical choices of the diffusion coefficient (such as the van-der-Waals model~\cite{DroZieSch23}), see Remark~\ref{cor:Ulrich}.

\subsection{Formulation of the problem}
We consider a two dimensional free boundary model for a keratocyte cell moving on a flat substrate with general nonlinear myosin diffusion including the van der Waals model. 

\medskip

The cell occupies a time-dependent domain $\Omega(t)\subset \R^2$ with a free boundary $\partial \Omega(t)$.
The velocity field of the cell $\mathbf{v}(\cdot, t): \Omega(t) \to \mathbb{R}^2$ is related to the scalar stress $\sigma(\cdot,t)$ through Darcy's law
\begin{equation}\label{eq:darcy}
  \mathbf{v} = \frac{1}{\zeta}\nabla \sigma \quad \text{in }\Omega(t)
\end{equation}
with drag coefficient $\zeta$. The stress is modeled by the constitutive law
\begin{equation}\label{eq:const_law}
    \mu \text{div } \mathbf{v} = \sigma - \chi m \quad \text{in } \Omega(t),
\end{equation}
where $\mu $ is the effective viscosity of the actin-myosin gel, $m(\cdot,t)$ is the density of myosin motors, and $\chi$ is the contractility per myosin motor protein. 
We impose the nonlocal elastic boundary condition
\begin{equation}\label{eq:nonloc_bc}
\sigma = -\tilde\gamma H -k\frac{|\Omega(t)| - |{\Omega_{\mathrm{ref}}}|}{|{\Omega_{\mathrm{ref}}}|} \quad \text{on } \partial \Omega(t),
\end{equation}
where $|\Omega|$
{is the area of the current domain, $\Omega_{\mathrm{ref}}= B_{R_\mathrm{ref}}$ is the reference ball domain of the cell of radius $R_{\mathrm{ref}}$,}
 and $k$ is the inverse elasticity coefficient of the cell membrane. 
The boundary velocity is related to the flow field via the kinematic boundary condition 
\begin{equation}\label{eq:kinematic_bc}
    V_\nu = \mathbf{v} \cdot \nu \quad  \text{on } \partial \Omega(t),
\end{equation}
stating that the free boundary is transported by the velocity field $\mathbf v$. Here $V_\nu$ denotes the normal velocity of the boundary, cf.~classical Hele-Shaw in fluids. 
The main novelty of this model lies in the advection-diffusion equation for the myosin motor density
\begin{equation}\label{eq:adv_diff_eq}
    \partial_t m + \text{div}(m \mathbf{v}) = \text{div} (\mathcal D D(m) \nabla m) \quad \text{on } \Omega(t),
\end{equation}
where we introduce the (nondimensional) nonlinear diffusion coefficient $D(m)$ and diffusivity constant $\mathcal D$. 
The case $D(m) = 1$ corresponds to the case of linear diffusion studied in \cite{BerRybSaf2022,RybBer23}. 
Our results hold for a general form of the nonlinear diffusion coefficient and we also show how the results can be applied to a particular $D(m)$, such as the van der Waals model in \cite{DroZieSch23}. 
The system is complemented with the no-flux boundary condition
\begin{equation}\label{eq:no_flux_bc}
\partial_\nu m = 0\quad \text{on } \partial \Omega(t),
\end{equation}
ensuring the conservation of total myosin mass
\begin{equation}
\int_{\Omega(t)} m(x,t) \, dx = M \quad \text{for all } t\geq 0.
\end{equation}

\medskip

Following the non-dimensionalization in \cite{RecPutTru15} we derive three non-dimensional parameters $K=\frac{k}{\zeta \mathcal D}$ (the Peclet number), $P=\frac{\chi M}{ k{\pi R_{\mathrm{ref}}^2}}$ (myosin contractility per motor), $Z=\frac{\mu}{\zeta {\pi R_{\mathrm{ref}}^2}}$ (arising from the ratio of dissipative to friction length scales), as well as non-dimensional surface tension $\gamma = {\frac{\tilde \gamma}{k R_{\mathrm{ref}}}}$.
In their non-dimensional form, the governing equations for the 2D free-boundary model are
\begin{align}
  Z\Delta\sigma        &= \sigma - Pm        && \text{in } \Omega(t), \label{eq:1_2d}\\
  \partial_t m         &= \operatorname{div} \bigl(D(m)\nabla m - K m \nabla\sigma\bigr)
                                                && \text{in } \Omega(t),\label{eq:2_2d}\\
  \partial_\nu m       &= 0                  && \text{on } \partial\Omega(t),\label{eq:3_2d}\\
  \sigma               &= -\gamma H + 1 - |\Omega(t)|
                                                && \text{on } \partial\Omega(t),\label{eq:4_2d}\\
  K\partial_\nu\sigma  &= V_\nu              && \text{on } \partial\Omega(t), \label{eq:5_2d}
\end{align}
{with the integral constraint 
\begin{equation}\label{eq:myosin_mass}
  \int_{\Omega(t)} m(x,t)\,dx =1 \quad \text{for all } t\geq 0.
\end{equation}
}
\medskip

{In the physical range of parameters}, the system admits a simple stationary solution corresponding to a radially symmetric resting cell
\begin{equation}\label{eq:steady_state}
\Omega(t)= B_{R_0},\quad m(x,t) = m_0 = \frac1{\pi R_0^2}, \quad \sigma(x,t) = \sigma_0=-\frac{\gamma}{R_0}+1-\pi R_0^2,
\end{equation}
where $R_0$ is the largest positive solution of $0=- \frac{\gamma}{R_0}+1-\pi R_0^2  - \frac{P}{\pi R_0^2}$, ensuring the compatibility in \eqref{eq:1_2d}. 
For $\gamma=0$, the exact value is easily calculated as $R_0=R_0(P)= \frac1{\sqrt{\pi}} \big(\frac12 +(\frac14 - P )^\frac12\big)^\frac12$.  Note that the two negative solutions are unphysical and we expect the smaller positive root to give rise to an unstable steady state, as was observed in the 1D case~\cite{RecPutTru15}. Note also that these steady states do not depend on $K$.

\medskip

Observe that this system has several interesting features. First, note that boundary condition \eqref{eq:4_2d} is non-local. It was introduced in \cite{BerRybSaf2022} for a 2D model and generalized the nonlocality in the 1D model  \cite{RecTruThi13,RecPutTru15, PutRecTru18}. This boundary condition was further mathematically studied in \cite{RybBer23} and in \cite{BerRybSaf2022}. This nonlocality was shown to result in the non-self-adjointness (NSA) of the linearized operator for the problem \eqref{eq:1_2d}--\eqref{eq:5_2d}. It was shown in \cite{BerSafTru26} that due to the NSA properties the standard eigenvalues (eigenmodes) stability analysis does not apply and in particular eigenvectors may not span the entire phase space. The linear and nonlinear stability was established subsequently in \cite{ RybBer23, BerRybSaf2022,} based on the analysis of the resolvent operator. 

\medskip

Another notable feature of this model is the cross-diffusion Keller-Segel type term in \eqref{eq:2_2d} that may result in a blow-up which interacts with nonlinearity due to the moving and deformable free boundary. Also, classical free boundary techniques based on the conformal mapping of $\Omega(t)$ to the unit disk cannot be applied here because the nonlinear PDE \eqref{eq:1_2d}--\eqref{eq:5_2d} is not conformally invariant unlike the classical Hele-Shaw problem, where the pressure is harmonic.

\subsection{Main results}

A central goal in the study of the system \eqref{eq:1_2d}--\eqref{eq:5_2d} is to understand the bifurcation from the stationary solutions to the traveling wave solutions (TWs). 
First, we prove the existence of traveling wave solutions 
and their bifurcation via the Crandall-Rabinowitz theorem~\cite{crandall1971bifurcation}. 
{Proposition~\ref{thm:bif_and_tw_existence} (below)} states that the bifurcation from the steady state~\eqref{eq:steady_state} occurs at the critical Peclet number $K=K_0$ that is the solution of the transcendental equation
\begin{equation}\label{eq:Bif_cond}
 P m_0 - \frac{D(m_0)}{K_0}\frac{J_1(\alpha)}{\alpha J'_1(\alpha)}=0,
\end{equation}
where $J_1$ is the first Bessel functions of the first kind and
\begin{equation}\label{eq:def_alpha}
\alpha=\alpha(K_0)=\frac{R_{0}}{\sqrt{Z}}\,
\sqrt{\frac{P\,K_{0}\,m_{0}}{D(m_{0})}-1}.
\end{equation}
The {Proposition} applies in this general context, we only need to impose the following {non-resonance condition} 
\begin{equation}\label{eq:non-resonance}
J_n'(\alpha) \neq 0,\ n\geq 2,
\end{equation}
 and non-degeneracy
\begin{equation}\label{eq:tranversality_cond}
 \frac{R_0^2}{Z} \frac{1}{\alpha J_1'(\alpha)} \int_0^1 s J_1(\alpha s)^2\,ds \neq 1
\end{equation}
 conditions on our physical parameters $P$, $Z$, and the diffusion coefficient $D(m)$.

The non-degeneracy condition ensures that the two solution branches are non-tangential at the bifurcation point and appears in our analysis of the transversality condition in the Crandall-Rabinowitz theorem~\cite{crandall1971bifurcation}.
{
\begin{proposition}[Existence and bifurcation of TWs]\label{thm:bif_and_tw_existence}
   Let $P,Z,\gamma>0$ and let $(R_0, m_0, \sigma_0)$ be the homogeneous stationary solution of \eqref{eq:1_2d}--\eqref{eq:5_2d} given by \eqref{eq:steady_state}. 
   Let $K_0$ be the critical value of the bifurcation parameter $K$ given by the transcendental equation \eqref{eq:Bif_cond}. 
   Let the nonlinear diffusion coefficient $D =D(m)$ be positive and four times continuously differentiable at $m_0$. {Moreover, assume that the non-resonance condition \eqref{eq:non-resonance} and the non-degeneracy condition \eqref{eq:tranversality_cond} are satisfied.}

   Then, there exists an interval $I = (-\varepsilon, \varepsilon)$, a function $R: I \times S^1 \to \R $ such that $R(V, \cdot)$ parametrizes  the boundary $\partial \Omega_V$ of a domain $\Omega_V$, and three
   functions 
   \begin{equation*}
     m\colon \{(V,x)\colon V\in I ,\, x\in  \Omega_V\} \to \R,\ \sigma: \{(V,x)\colon V\in I ,\, x\in  \Omega_V\}\to \R,\ K\colon I\to \R,
   \end{equation*}
   such that, for all $V \in I$, the tuple
   \begin{equation}
       (\Omega_V+Vt\,e, m(V,x-Vt\,e), \sigma(V,x-Vt\,e))
   \end{equation}
   is a traveling wave solution to the system \eqref{eq:1_2d}--\eqref{eq:5_2d} with Peclet number $K = K(V)$.
    This one-parameter family of traveling waves bifurcates from the steady state \eqref{eq:steady_state} at $K(0)=K_0$ and $V=0$. Moreover, this family of solutions is three times continuously differentiable in $V$.
\end{proposition}
}

Proposition~\ref{thm:bif_and_tw_existence} allows us to expand the Peclet number $K$ for small velocities $V$ around the bifurcation point
\begin{align}\label{eq:expansionKmainresults}
 K= K_0 + K_1V + K_2V^2 +\ldots .
\end{align}

The second-order coefficient $K_2$ -- the curvature of the bifurcation curve at the bifurcation point -- is the protagonist of this work as it determines the bifurcation type, cf.~Fig.~\ref{fig:K_2_for_eA}.
Note that $K_0$ is the location of the bifurcation point and by the symmetry $V \mapsto -V$ we have $K_1=0$.

\begin{figure}[h!]
    \centering
\includegraphics[width=.7\linewidth]{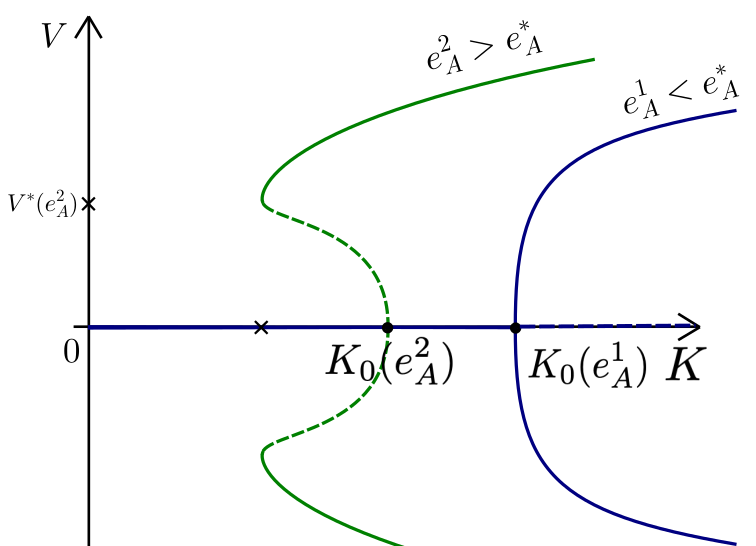}
    \caption{Bifurcation diagram for traveling
waves (TWs). Change of bifurcation type from supercritical (blue) to inverse (green) pitchfork at critical $e_A = e_A^\ast$. Bifurcation type depends on $K_2$: blue - $K_2>0$, green $K_2<0$.}
    \label{fig:K_2_for_eA}
\end{figure}

\medskip

The following main result of this paper provides a rigorous mathematical derivation of an explicit formula that controls the transition between sub- and supercritical bifurcation.  

\begin{theorem}[Bifurcation type]\label{th:K2_f-la}
	Consider the system \eqref{eq:1_2d}--\eqref{eq:5_2d} with given physical parameters $P, Z,\gamma>0$, and a positive four times continuously differentiable diffusion coefficient~$D = D(m)$. Let $m_0=m_0(P,\gamma)$ be the constant steady state~\eqref{eq:steady_state} and assume that our physical parameters satisfy relations~\eqref{eq:non-resonance},\eqref{eq:tranversality_cond}.
	
	Then $K_2$ in \eqref{eq:expansionKmainresults} is given by the explicit formula
\begin{equation}\label{eq:k2_fla_theorem}
        K_2 = A_1 \frac{D''(m_0)}{D(m_0)^2} + A_2 \frac{D'(m_0)^2}{D(m_0)^3} + A_3 \frac{D'(m_0)}{D(m_0)^2} + A_4 \frac{1}{D(m_0)},
    \end{equation}
  where  $A_i= A_i(P,Z, \gamma)$, $i=1,\ldots, 4$, are independent of $D(m)$ and are explicitly given by~\eqref{eq:A_i_fla}.
\end{theorem}

For a given set of physical parameters $P$, $Z$, $\gamma$ and diffusion coefficient $D(m)$, this formula allows to determine the bifurcation type and find the critical value of the bifurcation  parameter.
Indeed, our general result, Theorem~\ref{th:K2_f-la}, provides insight into a wide range of relevant physical models.
We illustrate this in our next main result, in which we apply our general formula \eqref{eq:k2_fla_theorem} to {{a 2D version of}} the van der Waals model for myosin~\cite{DroZieSch23}, and precisely predict the change of bifurcation that was previously observed numerically in 1D by Drozdowski et al.~\cite{DroZieSch23}. 

\medskip
\begin{figure}[h!]
    \centering
    \includegraphics[width=0.8\linewidth]{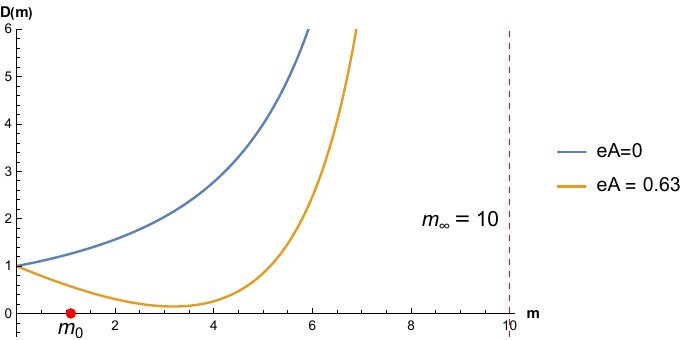}
    \caption{Graph of $D(m)$ given by~\eqref{eq:diff_coeff} for different choices of $e_A$ and $m_\infty = 10$. {Note the negative slope of $D(m)$ at $m=m_0$.}}
    \label{fig:D(m)}
\end{figure}
In this model,  the diffusion coefficient is of the form
\begin{equation}\label{eq:diff_coeff}
	D(m) = \frac{m_\infty^2}{(m_\infty-m)^2}- e_A m, 
\end{equation}
where $m_\infty$ is the saturation concentration of myosin, and $e_A$ is the cooperative binding ratio, see Fig.\ \ref{fig:D(m)}. Here $e_A$ must be smaller than $e_A^{crit}=\dfrac{27}{4m_\infty}$, so that $D(m)>0$ for all $m \geq 0$. {Here we take the values of the surface tension parameter $\gamma$ and  $m_\infty$ from experimental data in~\cite{barnhart2015balance,DroZieSch23} and a wide range of parameters $P,Z$, e.g.~\cite{barnhart2017adhesion}.}

\begin{figure}
    \centering
\includegraphics[width=.7\linewidth]{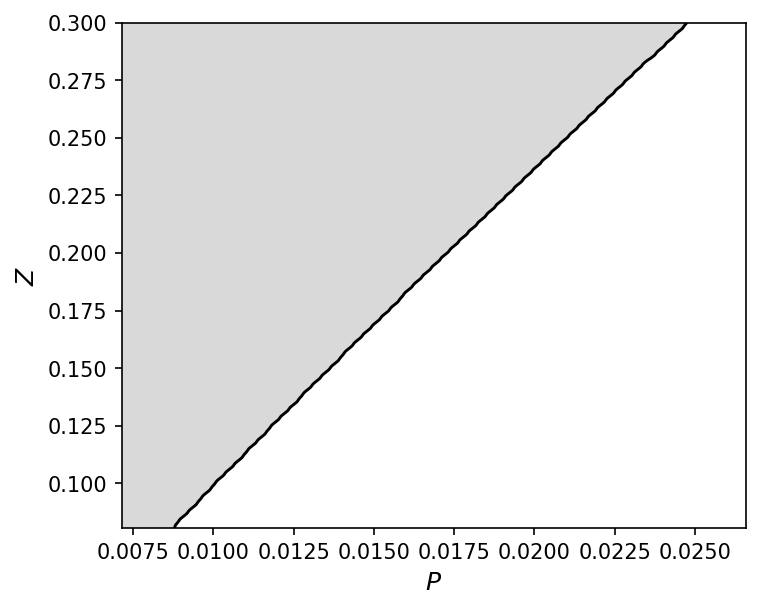}
    \caption{Shaded region in the \((P,Z)\)-plane is where the bifurcation type changes, for fixed \(\gamma=0.0025\). Here \(D(m)\) is given by \eqref{eq:diff_coeff} with \(m_\infty=10\).}
    \label{fig:change_of_bifurcation}
\end{figure}

\medskip
{
\begin{remark}[Change of bifurcation type in van der Waals model]\label{cor:Ulrich}
    Consider problem \eqref{eq:1_2d}--\eqref{eq:5_2d} with fixed \(
\gamma=0.0025, \)
and let the nonlinear diffusion coefficient \(D(m)\) be given by~\eqref{eq:diff_coeff} with \(m_\infty=10\).
Let \(K_2=K_2(e_A;P,Z)\) be given by~\eqref{eq:k2_fla_theorem}.

Numerical computations show that there exists a region
\(
\mathcal{R}\subset \mathbb{R}^2_{+}
\)
in the \((P,Z)\)-plane, displayed in Figure~\ref{fig:change_of_bifurcation}, such that for every \((P,Z)\in \mathcal{R}\) there exists a critical value
\[
e_A^\ast=e_A^\ast(P,Z)\in (0,e_A^{\mathrm{crit}})
\]
satisfying
\[
K_2\bigl(e_A^\ast;P,Z\bigr)=0.
\]
Moreover, the bifurcation from the stationary state to a traveling wave solution occurs
\begin{enumerate}[label=(\roman*)]
    \item via a supercritical pitchfork if \(e_A<e_A^\ast(P,Z)\), and
    \item via a subcritical pitchfork if \(e_A^\ast(P,Z)<e_A<e_A^{\mathrm{crit}}\).
\end{enumerate}
\end{remark}

Outside the region \(\mathcal{R}\) defined in Remark~\ref{cor:Ulrich}, no change of bifurcation type was observed numerically for \(e_A<e_A^{crit}\).

\begin{remark}
  {
  It is instructive to follow our proof along the 1D counterpart of \eqref{eq:1_2d}--\eqref{eq:5_2d}. In this case some computations simplify and one gets an analogous formula for $K_2$. In case of the van der Waals model this formula gives $e_A^\ast=0.5990\ldots$, which agrees with the numerical study in~\cite{DroZieSch23}. This is visualized in Fig.~\ref{fig:K2_of_eA}.}
\end{remark}
\begin{remark}
We verify numerically that the formula~\eqref{eq:k2_fla_theorem} in 2D is correct. This was done in two ways by i) computing the convergence rates of the residuals as $V \to 0$; ii) verifying that the $K_2$ we obtain indeed makes the third-order expansion of the system~\eqref{eq:1_2d}-\eqref{eq:myosin_mass} solvable. We refer to Appendix~\ref{app:numerics} for the details.
\end{remark}

Remark~\ref{cor:Ulrich} is visualized in Fig.~\ref{fig:K_2_for_eA}, Fig.~\ref{fig:change_of_bifurcation}, and {Fig.~\ref{fig:K2_of_eA}}.
Note that Fig.~\ref{fig:D(m)} shows that for $e_A= 0.63$ the diffusion coefficient $D(m)$ decreases at $m=m_0$ which is necessary for the subcritical pitchfork bifurcation in view of the signs of the coefficients $A_i$ in Remark~\ref{cor:Ulrich}. This again confirms the numerical findings in~\cite{DroZieSch23}.

\begin{figure}[h!]
    \centering
    \includegraphics[width=0.95\linewidth]{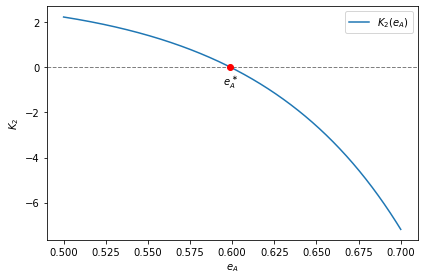}
    
    \caption{Graph of $K_2(e_A)$ for $D(m)$ in 1D for $P=0.1,\, Z=1.25,\,m_\infty =10$.}
    \label{fig:K2_of_eA}
\end{figure}
}
\medskip 
{
The key questions in biophysics studies of the motions of cells are transitions between steady and motile states. 
In particular, this transition can occur in two ways: by continuous increase of speed of resting cells or by ``jerking'' from zero to high velocity. Both of these scenarios can be experimentally observed, e.g., in \cite{lacayo2007emergence}. Such transitions can be described by supercritical and subcritical bifurcation, respectively. Furthermore, cell motility can be viewed as a competition between aggregation and diffusion of myosin, which in broad terms can happen in two scenarios: i) the diffusion dominates and the cell is at rest, or ii) the aggregation is strong enough to break the symmetry of myosin distribution inside the cell and the resulting asymmetry drives the motion.

In this context, Theorem~\ref{th:K2_f-la} and Remark~\ref{cor:Ulrich} can be interpreted as follows. Formula~\eqref{eq:k2_fla_theorem} for $K_2$ from Theorem~\ref{th:K2_f-la} shows that
the type of the bifurcation changes to subcritical pitchfork 
due to the fact that the diffusion starts decaying fast enough,
when $-D'(m_0)\gg 1$. This decay of diffusion facilitates rapid aggregation of myosin at the rear of the cell, which in turn leads to a jump in the velocity (``jerking'' behavior). In contrast, when $-D'(m_0)\ll 1$, cells behave similarly to the linear diffusion case $D'(m_0)=0$, where cells accelerate slowly. Remark~\ref{cor:Ulrich} addresses a specific form of $D(m)$ that appears in biophysics studies~\cite{DroZieSch23}. Then $-D'(m_0)=-\frac{2m_\infty^2}{(m_\infty-m_0)^3} +e_A$ is controlled by $e_A$ and for $e_A>e_A^*$ the ``jerking'' occurs, whereas for $e_A<e_A^*$ slow acceleration occurs. Note that our results allow for the explicit calculation of $e_A^*$ in terms of physical parameters and can be used in biophysics studies, both theoretical and experimental.

Finally, we mention that subcritical bifurcation in biophysics is known to be related to bistability \cite{DroZieSch23}, that is, the coexistence of steady and motile states for the same range of physical parameters. In experimental setting the switch between these two states can be done by external cues, such as an optogenetic activation that induces motion \cite{hadjitheodorou2021directional}. Our work raises a very interesting and challenging  question of establishing bistability mathematically. 
}

\subsection{Ideas of the proofs}\label{sec:ideas_of_proofs}

The proof of our main result, Theorem~\ref{th:K2_f-la}, rests on the asymptotic expansion of the traveling wave solutions for small velocities \eqref{eq:V_exp}-\eqref{eq:O_exp} provided by {Proposition}~\ref{thm:bif_and_tw_existence}. It is convenient to change coordinates into a moving frame with velocity $Ve$, in which the time-dependence is eliminated, so that the expansion reads
\begin{align}
	\sigma(r, \theta, V) &= \sigma_0 +\sigma_1(r,\theta) V+ \sigma_2(r,\theta)V^2+ \sigma_3(r,\theta)V^3+o(V^3)\label{eq:V_exp}\\
	m(r, \theta, V) &= m_0 + m_1(r, \theta) V + m_2(r, \theta) V^2+ m_3(r, \theta) V^3 + o(V^3)\\
	K(V) &= K_0 + K_1 V + K_2 V^2 + o(V^2)\\
    \rho(\theta,V) &= R_0 + \rho_1(\theta) V +\rho_2(\theta) V^2 + \rho_3(\theta) V^3 + o(V^3), \label{eq:O_exp}
\end{align}
where we express the coefficients in polar coordinates.
Due to the regularity provided by {Proposition}~\ref{thm:bif_and_tw_existence}, we can match coefficients of $V$ in this expansion to rigorously derive the PDE systems solved by the respective terms.
Our protagonist $K_2$ does not appear before the third-order expansion of the system~\eqref{eq:1_2d}--\eqref{eq:5_2d}, see Table~\ref{tab:expansion}. 

\begin{table}[h!]
\begin{center}
    \begin{tabular}{|c|c| c| c | c |}
      \hline
      & \multicolumn{4}{|c|}{Unknowns arising at $V^i$}\\
      \hline
       $i$ & 0 &1 & 2 & 3\\
      \hline
      $\sigma$ & $\sigma_0 = const$ & $\sigma_1$ & $\sigma_2$ & $\sigma_3$ \\
      \hline
      $m$ & $m_0 = const$ & $m_1$ & $m_2$ & $m_3$ \\
      \hline 
      $\rho$ & $R_0$ & $\rho_1=0$ & $\rho_2$ & $\rho_3 =0$ \\
      \hline
      $K$ & - & $K_0$ & $K_1=0$ & $K_2$ \\
      \hline
    \end{tabular}
\end{center}
\caption{Unknowns arising in expansion}\label{tab:expansion}
\end{table}

While the first-order system for $m_1$ and $\sigma_1$ can be solved analytically, the higher-order systems are no longer amenable to such a direct analysis.
Instead, we observe that the additional kinematic boundary condition~\eqref{eq:5_2d} in our free boundary problem makes this third-order expansion of the PDE system overdetermined and the formula for $K_2$ can be viewed as a compatibility condition, similar to the Fredholm alternative. 
While the Fredholm alternative applies in the stiff limit, the free boundary makes its application difficult due to the additional kinematic boundary condition.
Instead, we construct a test function that extracts a relevant mode from the third-order expansion of the PDE system~\eqref{eq:1_2d}--\eqref{eq:5_2d}.

The construction of this test function is as follows. We combine the third-order expansions of~\eqref{eq:1_2d} and~\eqref{eq:2_2d} suitably to an equation of the form
\begin{equation}\label{eq:proof_combined_eqs}
    \Delta \sigma +  \Big(\dfrac{\alpha}{R_0}\Big)^2\sigma = f(m_0,\sigma_0,\ldots,m_2,\sigma_2).
\end{equation} 
Our test function satisfies the formal adjoint PDE to this equation with constant non-homogeneous Dirichlet boundary conditions
\begin{equation}
\label{eq:testfunction}
\begin{cases}
    \Delta u + \Big(\dfrac{\alpha}{R_0}\Big)^2 u =0 & \text{in }B_{R_0},\\
    u= \cos \theta &  \text{on }\partial B_{R_0},
\end{cases}
\end{equation}
which in the one-dimensional case simply corresponds to a sine function.

Testing our derived PDE~\eqref{eq:proof_combined_eqs} with this test function then gives us an explicit formula for $K_2$.

\medskip
{
Finally, we mention that the proof of Proposition \ref{thm:bif_and_tw_existence} is based on the Crandall-Rabinowitz bifurcation theorem.  The application of this theorem in free-boundary problems is typically based on the existence of a solution operator which solves the PDE with all but one boundary condition. We propose an alternative approach which does not rely on a solution operator and allows one to bypass its construction and directly check transversality and simple zero eigenvalue conditions in a vectorial functional setting. {We remark} that constructing a solution operator in some applications (e.g. in our case) may be more difficult than the direct approach proposed here. We note that our approach involves several subtle analytical issues, such as choosing function spaces that provide sufficient regularity, introducing an auxiliary scalar variable to recover the Fredholm index-zero structure of the linearized operator, and imposing non-resonance and non-degeneracy conditions on the physical parameters.
}

\section*{Acknowledgments}

The authors are grateful to Ulrich Schwarz,  Falko Ziebert, Oliver Drozdowski, {and Alexander Mogilner} for many fruitful discussions and pointing out their work. 

\section*{Funding Declaration}
This project has received funding from Deutsche Forschungsgemeinschaft (DFG, German Research Foundation) 
under Germany's Excellence Strategy -- EXC-2047/1 -- 390685813
and from the Research Training Group 2339 IntComSin – Project-ID 321821685.
O.K.\ gratefully acknowledges summer support for extended stays at the Hausdorff Center for Mathematics (HCM), University of Bonn and IntComSin at the University of Regensburg. 

The work of L.B.\ and O.K.\ was partially supported by the National Science Foundation
grants DMS-2404546 and DMS-2005262. The authors also acknowledge the hospitality and support of Heidelberg University during their visits, where L.B.\ was Romberg Visiting Professor for a one-month sabbatical stay. {The authors acknowledge support of the Institut Henri Poincaré (UAR 839 CNRS-Sorbonne Université), and LabEx CARMIN (ANR-10-LABX-59-01)}

\section{Proof of Theorem \ref{th:K2_f-la}: Change of bifurcation type}

The proof of the theorem is based on the following three lemmas that concern the coefficients of the first- to third-order expansions in~\eqref{eq:V_exp}--\eqref{eq:O_exp} of the solutions of our PDE system around the steady state for small velocity.
These lemmas extract the precise dependence on the diffusion coefficient in the expansion.

\begin{lemma}[First-order expansion]\label{lem:first_ord}
    The first-order coefficients of the TW solution of~\eqref{eq:1_2d}--\eqref{eq:5_2d} and the bifurcation point $K_0$ are given by    \begin{align}\label{lem:1stord_2d}
        m_1(x)&= m_{11}(r) \cos \theta= \frac{1}{D(m_0)} \hat m_{11}(r) \cos \theta ,\\
        \sigma_1(x) &= \sigma_{11}(r) \cos \theta = \frac{1}{D(m_0)} \hat \sigma_{11}(r) \cos \theta,\\
        K_0 &= D(m_0) \hat K_0,
    \end{align}
where $\hat m_{11}(r)$, $\hat\sigma_{11}(r)$, and $\hat K_0$ are independent of the choice of $D(m)$.
\end{lemma}

\begin{lemma}[Second-order expansion]\label{lem:2ndorder}
    The second-order coefficients of the TW solution of~\eqref{eq:1_2d}--\eqref{eq:5_2d} are given by  \begin{align}\label{lem:2ord1}
		m_2(x) &= m_{20}(r)+m_{22}(r)\cos (2\theta), \\
        \sigma_2(x) &= \sigma_{20}(r)+\sigma_{22}(r)\cos (2\theta),\\
        \rho_2(\theta) &= \rho_{20}+\rho_{22}\cos (2\theta),
		\label{lem:2ord3}
	 \end{align}	 
	 where all coefficients $m_{20}(r), m_{22}(r), \ldots$ depend on $D(m)$ in the same way via
     \begin{equation}
         m_{20}(r) =\frac{1}{D(m_0)^2} m_{20A}(r)+ \frac{D'(m_0)}{D(m_0)^3}m_{20B}(r).
     \end{equation}
\end{lemma}

As described in the previous section, the third-order expansion is crucial as it determines the curvature of the curve of bifurcating solutions. However, this system is too complicated to be solved. Nevertheless, our test method gives us an explicit formula by extracting a relevant mode from the PDE.

\begin{lemma}[Third-order expansion and $K_2$]\label{lem:3rdorder}
	The curvature $K_2$ of the bifurcating curve of TW solutions of~\eqref{eq:1_2d}--\eqref{eq:5_2d} at the bifurcation point $K_0$ is given by 
    \begin{align}\label{eq:K2_general}
    &K_2 
    = \bigg(-Z\frac{R_0}{K_0^2}+\frac{Pm_0}{D(m_0)}\int_0^{R_0}\sigma_{11}\left(r\right) U(r)rdr\bigg)^{-1} \times 
    \\&\times \Bigg\{\int_0^{R_0}  -P\!\left[
      \bigl(\rho_{20}+ \frac{1}{2}\rho_{22}\bigr) \left(\frac{K_0 m_0}{D(m_0)} \sigma''_{11}(R_0)-m_{11}''(R_0)\right){+}\frac{\rho_{22}}{R_0^2}m_{11}(R_0)
      \right]r U(r) rdr\nonumber \\
      &- \int_0^{R_0}  \frac{P}{D(m_0)} \left(\frac{1}{2r} \int_0^r s^2 f(s)\,ds +\frac{r}{2}\int_{r}^{R_0}f(s)\,ds+\frac{r}{2R_0^2}\int_{0}^{R_0} s^2f\left(s\right)ds\right)U(r) rdr \nonumber \\
    &\qquad\qquad \quad +Z R_0\left(\rho_{20}+\frac{\rho_{22}}{2}\right)\left(\sigma_{11}''(R_0)-\frac{\alpha/R_0}{K_0} \frac{J_1'(\alpha)}{J_1(\alpha)}\right) \Bigg\},\nonumber
\end{align}
where $u(r,\theta)=U(r)\cos \theta$ is the test function {given by explicit solution of}~\eqref{eq:testfunction} ({in Bessel functions}), and $f(s)$ is given by~\eqref{eq:f(s)}. Moreover, the function $U(r)$ is independent of $D(m)$, and the function $f(r)$ can be represented as 
\begin{equation*}
    f(r) = \frac{D''(m_0)}{D(m_0)^3}f_1(r)+\frac{D'(m_0)^2}{D(m_0)^4}f_2(r)+\frac{D'(m_0)}{D(m_0)^3}f_3(r)+\frac{1}{D(m_0)^2}f_4(r),
\end{equation*}
where $f_1(r), \ldots, f_4(r)$ are given by \eqref{eq:f1}--\eqref{eq:f4}.
\end{lemma}

The simple proofs of the lemmas are deferred to Appendix~\ref{app:lemmas}.

\begin{proof}[Proof of Theorem~\ref{th:K2_f-la}]
    We use formula~\eqref{eq:K2_general} provided by Lemma~\ref{lem:3rdorder}.
    To reveal the precise dependence of the coefficients $\tilde A_1, \ldots, \tilde A_4$ on the diffusion coefficient $D$, we use Lemmas~\ref{lem:first_ord} and~\ref{lem:2ndorder}.

First, let us show that the first term  of \eqref{eq:K2_general} scales like $1/D(m_0)^2$.
Indeed,
\begin{align}\label{eq:denum_exp}
    &-Z\frac{R_0}{K_0^2}+\frac{Pm_0}{D(m_0)}\int_0^{R_0}\sigma_{11}(r) U(r)r\,dr  \nonumber
    \\
    &=\frac{1}{D(m_0)^2} \left( -\frac{ZR_0}{\hat K_0^2}+Pm_0 \int_0^{R_0}\hat\sigma_{11}(r)U(r)r\,dr\right)
   =
  \frac{1}{D(m_0)^2} A_0,
\end{align}
where $A_0$ is independent of $D(m_0)$. 

A similar computation for the second term in \eqref{eq:K2_general} shows that it can be expressed as
\begin{equation}\label{eq:num_exp}
   \tilde A_1 \frac{D''(m_0)}{D(m_0)^4} + \tilde A_2  \frac{D'(m_0)^2}{D(m_0)^5} +\tilde  A_3  \frac{D'(m_0)}{D(m_0)^4} +\tilde  A_4  \frac{1}{D(m_0)^3},
\end{equation}
    where 
    \begin{align}\label{eq:tildeA1}
        &\tilde A_1=-P \int_0^{R_0} \left(\frac{1}{2r} \int_0^r s^2 f_1(s)ds +\frac{r}{2}\int_{r}^{R_0}f_1(s)ds+\frac{r}{2R_0^2}\int_{0}^{R_0} s^2f_1(s)ds\right)U(r) r\,dr,
    \end{align}

    \begin{align}
        &\tilde A_2 =-P \int_0^{R_0} \left(\frac{1}{2r} \int_0^r s^2 f_2(s)ds +\frac{r}{2}\int_{r}^{R_0}f_2(s)ds+\frac{r}{2R_0^2}\int_{0}^{R_0} s^2f_2(s)ds\right)U(r) r\,dr,
    \end{align}

    \begin{align}
      &\tilde A_3 =-P \int_0^{R_0} \left(\frac{1}{2r} \int_0^r s^2 f_3(s)ds +\frac{r}{2}\int_{r}^{R_0}f_3(s)ds+\frac{r}{2R_0^2}\int_{0}^{R_0} s^2f_3(s)ds\right)U(r) r\,dr \nonumber\\
      & -\int_0^{R_0} P\!\left[
      \bigl(\rho_{20B}+ \frac{1}{2}\rho_{22B}\bigr) \left(\hat K_0 m_0 \hat \sigma''_{11}(R_0)-\hat m_{11}''(R_0)\right){+}\frac{\rho_{22B}}{R_0^2}\hat m_{11}(R_0)
      \right]r U(r) r\,dr \nonumber \\
      &\qquad \quad +Z R_0\left(\rho_{20B}+\frac{\rho_{22B}}{2}\right)\left(\hat \sigma_{11}''(R_0)-\frac{\alpha/R_0}{\hat K_0} \frac{J_1'(\alpha)}{J_1(\alpha)}\right)
    \end{align}
\begin{align}\label{eq:tildeA4}
    &\tilde A_4 =-P \int_0^{R_0} \left(\frac{1}{2r} \int_0^r s^2 f_4(s)ds +\frac{r}{2}\int_{r}^{R_0}f_4(s)ds+\frac{r}{2R_0^2}\int_{0}^{R_0} s^2f_4(s)ds\right)U(r) r\,dr \nonumber\\
      & -\int_0^{R_0} P\!\left[
      \bigl(\rho_{20A}+ \frac{1}{2}\rho_{22A}\bigr) \left(\hat K_0 m_0 \hat \sigma''_{11}(R_0)-\hat m_{11}''(R_0)\right){+}\frac{\rho_{22A}}{R_0^2}\hat m_{11}(R_0)
      \right]r U(r) r\,dr \nonumber \\
      &\qquad \quad +Z R_0\left(\rho_{20A}+\frac{\rho_{22A}}{2}\right)\left(\hat \sigma_{11}''(R_0)-\frac{\alpha/R_0}{\hat K_0} \frac{J_1'(\alpha)}{J_1(\alpha)}\right)
\end{align}

Finally, plugging the  expressions~\eqref{eq:num_exp},~\eqref{eq:denum_exp} into the formula~\eqref{eq:K2_general} for $K_2$ we derive the final formula~\eqref{eq:k2_fla_theorem} with 
\begin{equation}\label{eq:A_i_fla}
    A_1 = \frac{\tilde A_1}{A_0}, A_2 = \frac{\tilde A_2}{A_0}, A_3 = \frac{\tilde A_3}{A_0}, A_4 = \frac{\tilde A_4}{A_0}, 
\end{equation}
where $A_0, \tilde A_i$ are given by \eqref{eq:denum_exp}, \eqref{eq:tildeA1}--\eqref{eq:tildeA4}. 
\end{proof}

\section{Proof of Proposition \ref{thm:bif_and_tw_existence}: Existence and bifurcation of TWs }

Our proof is based on the classical theorem by Crandall and Rabinowitz (C.-R.), see~\cite{crandall1971bifurcation}. Recall that this theorem establishes the existence of bifurcation of family of solutions $x=x(K)$ for the equation $F(x, K) =0$ under several conditions on $F$. We divide the proof of Proposition \ref{thm:bif_and_tw_existence} into four steps.

\begin{proof}[Proof of Proposition~\ref{thm:bif_and_tw_existence}]
{\bf Step 1.} {\it Functional setting for C.-R. theorem for our free-boundary problem with nonlinear diffusion.} In this step we perform the change of coordinates that maps the moving domain $\Omega(t)$ with free boundary to the unit disk and we compute the operator of the problem \eqref{eq:1_2d}--\eqref{eq:5_2d} in these new coordinates. Finally, we construct the function $F(x, K)$ for our application of the C.-R. Theorem given by \eqref{eq:F_nonlinear} below. 

We first perform the change of coordinates that transforms the problem \eqref{eq:1_2d}--\eqref{eq:5_2d} in the domain $\Omega(t)$ with free boundary to the following {time-independent }problem in the unit disk {by using the traveling wave ansatz}
\begin{subequations}\label{eq:2d_fixed_TW}
\begin{alignat}{2}
  \label{eq:1_2d_fixed_TW}
  Z\tilde{\Delta}\sigma
  &= \sigma - Pm,
  &\qquad& 0\le\rho\le1,\; 0\le\theta<2\pi\\[4pt]
  \label{eq:2_2d_fixed_TW}
  -\,V e_{1}\cdot\tilde{\nabla}m
  &= \tilde{\operatorname{div}}\bigl(D(m)\,\tilde{\nabla}m - K m \tilde{\nabla}\sigma\bigr),
  &\qquad& 0\le\rho\le1,\; 0\le\theta<2\pi\\[4pt]
  \label{eq:3_2d_fixed_TW}
  N[R]\,m
  &= 0,
  &\qquad& \rho=1,\; 0\le\theta<2\pi\\[4pt]
  \label{eq:4_2d_fixed_TW}
  \sigma
  &\,=\, 1-|\Omega(0)|-\gamma H[R],
  &\qquad& \rho=1,\; 0\le\theta<2\pi\\[4pt]
  \label{eq:5_2d_fixed_TW}
  0
  &\,=\, K\,N[R]\sigma - V e_{1}\!\cdot\!\nu[R],
  &\qquad& \rho=1,\; 0\le\theta<2\pi
\end{alignat}
\end{subequations}
with an integral constraint

\begin{equation}\label{eq:integral_const}
  \int_B m w[R]\, dx=1.
\end{equation}

The technical details of this coordinate change are presented in Section~\ref{app:free_to_fixed}, where the operators $\tilde \Delta, \tilde \nabla, \tilde{\text{div}}, N[R], H[R], \nu[R], \text{ and } w[R]$ are given by \eqref{eq:tilde_grad}-\eqref{eq:h[r]}. In short, the coordinate change is the combination of shift {to fix the center of the domain} and a Hanzawa transform {to map the free-boundary domain to the fixed unit disk}.

{
In order to construct the operator $F$, { such that its linearization $F_x$} satisfies the Fredholm index-zero property in the C.-R. theorem, we modify the equation \eqref{eq:2_2d} by adding an additional unknown $c$ to the right-hand side

\begin{equation}\label{eq:2_2d_mod}
  -\,V e_{1}\cdot\tilde{\nabla}m = \tilde{\operatorname{div}}\bigl(D(m)\,\tilde{\nabla}m - K m \tilde{\nabla}\sigma\bigr) + c.
\end{equation}
{
The scalar \(c\) is an auxiliary unknown introduced to recover the Fredholm index-zero structure of the linearized operator.
}
The next lemma shows that the problem with the extra unknown constant $c$ is equivalent to the original one in the sense that all solutions to the modified problem have $c=0$ and satisfy the original problem \eqref{eq:1_2d_fixed_TW}--\eqref{eq:5_2d_fixed_TW},\eqref{eq:integral_const}. We will use this modified formulation to prove the existence of solutions using the C.-R.\ theorem.
\begin{lemma}\label{lem:lemma4}
  There are no solutions $(m,\sigma, R,V,c)$ to \eqref{eq:1_2d_fixed_TW},\eqref{eq:2_2d_mod},\eqref{eq:3_2d_fixed_TW}--\eqref{eq:5_2d_fixed_TW},\eqref{eq:integral_const} with $c\neq0$. 
\end{lemma}
\begin{proof}
  Consider a solution $(m,\sigma, R,V,c)$ to the problem \eqref{eq:1_2d_fixed_TW},\eqref{eq:2_2d_mod},\eqref{eq:3_2d_fixed_TW}--\eqref{eq:5_2d_fixed_TW},\eqref{eq:integral_const}. If we pull this solution back to the original free boundary domain $\Omega$, equation \eqref{eq:2_2d_mod} reads 
  \begin{equation*}
      -\,V e_{1}\cdot{\nabla}m = {\operatorname{div}}\bigl(D(m)\,{\nabla}m - K m {\nabla}\sigma\bigr) + c.
  \end{equation*}
Now integrating this equation over the domain $\Omega$, and using analogues of boundary conditions \eqref{eq:3_2d_fixed_TW},\eqref{eq:5_2d_fixed_TW} on the free boundary, we get
\begin{align}
  &\int_\Omega \,V e_{1}\cdot{\nabla}m + {\operatorname{div}}\bigl(D(m)\,{\nabla}m - K m {\nabla}\sigma\bigr)\, dx= \\
  &\int_\Omega \,V e_{1}\cdot{\nabla}m\, dx + \int_{\partial \Omega} D(m)\partial_\nu m - K m \partial_\nu \sigma\, dS = \\
  &\int_\Omega \,V e_{1}\cdot{\nabla}m\, dx + \int_{\partial \Omega} -K m \left(\frac{1}{K} V e_1 \cdot \nu\right)\, dS =0
.\end{align}
Therefore, for any solution, we have $c=0$. 
\end{proof}
}

We next introduce the operator $F(x,K)$ parametrized by the scalar parameter $K$ that acts on {$x = ( m(\rho, \theta), R(\theta), V, c) \in X := X_m \times X_R \times \R \times \R$} via
\begin{equation}\label{eq:F_nonlinear}
  F\left(\begin{pmatrix}  m\left(\rho, \theta\right)\\
   R\left(\theta\right) \\
   V \\
  c \end{pmatrix}, K \right) =
  \begin{pmatrix}
    V e_1 \cdot \tilde{\nabla}m + \tilde{\text{div}}(D(m) \tilde{\nabla} m) - K \tilde{\text{div}}(m \tilde{\nabla}\sigma) + c\\
    N[R]m \\
    \sigma\left(1,\theta\right) - 1 + \frac{1}{2} \int_{0}^{2\pi} R(\theta)^2 d\theta  + \gamma H[R] \\
    \int_B m w_R\,dx -1
   \end{pmatrix},
\end{equation}
where $\sigma = S({x}, K)$ is defined via the solution operator to
\begin{subequations}
\begin{numcases}{}
    Z\tilde{\Delta}\sigma
  \,=\, \sigma - Pm,
  \hfill 0\le\rho\le1, 0\le\theta<2\pi \\
  0
  \,=\, K\,N[R]\sigma - V e_{1}\!\cdot\!\nu[R],
  \quad \hfill \rho=1, 0\le\theta<2\pi
\end{numcases}
\end{subequations}
We refer to \eqref{eq:Domain_space} in Section \ref{app:free_to_fixed} for the precise functional setting.  

Now our problem \eqref{eq:1_2d_fixed_TW},\eqref{eq:3_2d_fixed_TW}--\eqref{eq:5_2d_fixed_TW} can be written in the form $F(x, K)=0$ and we will next verify the conditions of the C.-R. theorem for the operator $F$
given by \eqref{eq:F_nonlinear}. The C.-R. theorem guarantees the existence of the bifurcation of TW solutions for small $V$ in a neighborhood of the trivial solution $x=x_0$ provided that
\begin{enumerate}[label=(\roman*)]
	\item \label{item:CR1} $F(x_0,K) = 0$ for all $K$ in a neighborhood of $K_0$.
    \item \label{item:CR2}$F_x, F_K, F_{x,K}$ exist and are continuous in a neighborhood of $(x_0,K_0)$.
    \item \label{item:CR3}$\mathrm{dim}( \ker F_x(x_0, K_0)) = 1$, i.e., there exists a simple zero eigenvector $x_1$ s.t.\ $F_x(x_0;K_0) x_1= 0$, and $\mathrm{codim}(\mathrm{Range}(F_x(x_0, K_0)))=1$.
    \item \label{item:CR4}$F_{x,K} (x_0, K_0) x_1 \not\in \mathrm{Range}(F_x(x_0;K_0))$.
\end{enumerate}
Condition~\ref{item:CR1} defines $x_0$ as a trivial solution. The first condition in~\ref{item:CR3} ensures the existence of a simple zero eigenvalue of the linearized operator $F_x$ and the second condition in~\ref{item:CR3} shows that the operator $F_x$ is Fredholm with index 0. 
The transversality condition \ref{item:CR4} ensures that the new nontrivial branch of solutions is non-tangential to the trivial branch.

Finally, we note  that condition~\ref{item:CR1} is easily verified, as the stationary solution \eqref{eq:steady_state} provides the trivial solution of \eqref{eq:F_nonlinear} given by
$
    x_0 = (
       m_0,
       R_0,
        0,
       0)
$.
Condition~\ref{item:CR2} is easily checked due to our regularity assumption on $D$.
In fact, we have $F\in C^3$ in a neighborhood of $(x_0,K_0)$, which will allow us to gain the additional regularity in $V$ stated in the Proposition. { The crucial aspect is that one needs to make sure that $R$ is uniformly bounded away from zero, which is guaranteed by restricting to a small neighborhood of $R_0$ in a suitable Sobolev space. Then, for such $R$, the mapping of the problem to the fixed domain does not have singularities and the map is sufficiently smooth. For the convenience of the reader, we provide details in Section~\ref{app:free_to_fixed}.}
We note that the nonlinear diffusion coefficient $D(m)$ causes significant differences in the verification of conditions~\ref{item:CR3} and~\ref{item:CR4} compared to the linear diffusion case, which we present in the next three steps.

{\bf Step 2. }{\it Establishing simple zero eigenvalue condition in~\ref{item:CR3}.}

To check the simple zero eigenvalue property, we start by computing the Frechet derivative of $F$ at the bifurcation point
\begin{equation}\label{eq:f_linearizeation}
    F_x\left(x_0, K_0 \right)\begin{pmatrix} m\left(\rho, \theta\right)\\ R\left(\theta\right) \\
       V \\
      c  \end{pmatrix}
    = 
    \begin{pmatrix} 
    D(m_0)\dfrac{1}{R_0^2}\Delta_{(\rho,\theta)} m - K_0 m_0 \dfrac{1}{R_0^2}\Delta_{(\rho,\theta)} \sigma +c \\
  \dfrac{1}{R_0} m_\rho\left(1, \theta\right)\\
  \sigma\left(1, \theta \right) +R_0 \int_{0}^{2\pi}R(\theta)\, d \theta-\gamma \dfrac{R(\theta)+R''(\theta)}{R_0^2}\\
R_{0}^{2}\int_B m \, dx+ m_0R_0\int_{0}^{2\pi} R(\theta)\,d\theta 
 \end{pmatrix} 
  ,\end{equation}
where $\sigma = S_x({x}, K)$ is a solution to
\begin{subequations}\label{eq:S_op}
\begin{numcases}{}
    \frac{Z}{R_0^2}{\Delta_{(\rho, \theta)}}\sigma
  \,=\, \sigma - Pm,
  \hfill \quad 0\le\rho\le1,  0\le\theta<2\pi \\
 \dfrac{K_0}{R_0} \sigma_\rho\left(1, \theta\right) =  V \cos \theta. 
\end{numcases}
\end{subequations}
and $\Delta_{(\rho,\theta)} u =  u_{\rho\rho}
+ \dfrac{u_{\rho}}{\rho}
+\dfrac{u_{\theta\theta}}{\rho^2} $ denotes the Laplace operator in polar coordinates. 
Thus we need to show the existence of a unique (up to a constant factor) solution
\begin{equation}\label{eq:x_form}
  x_1 = (
    m_1, R_1, V_1, c_1)^T
\end{equation} of 
\begin{subequations}\label{eq:simple_ev_sys}
    \begin{numcases}{}
         D(m_0)\Delta_{(\rho,\theta)} m_1 - K_0 m_0 \Delta_{(\rho,\theta)} \sigma_1 +R_0^2 c_1 =0 \label{eq:zero_ev_a}\\
                  \frac{Z}{R_0^2}{\Delta_{(\rho, \theta)}}\sigma_1 = \sigma_1 - Pm_1,\label{eq:zero_ev_b}\\
         m_{1,\rho}(1, \theta) = 0 \label{eq:zero_ev_c}\\
         \sigma_1(1,\theta)  +R_0 \int_{0}^{2\pi}R_1(\theta)\, d \theta-\gamma \dfrac{R_1(\theta)+R_1''(\theta)}{R_0^2} =0\label{eq:zero_ev_d} \\
         R_{0}^{2}\int_B m_1\, dx + m_0R_0\int_{0}^{2\pi} R_1(\theta)\,d\theta  =0 \label{eq:zero_ev_e}\\
 \dfrac{K_0}{R_0} \sigma_{1,\rho}\left(1, \theta\right) =  V_1 \cos \theta. \label{eq:zero_ev_f}
    \end{numcases}
\end{subequations}

First, we show that $c_1=0.$ As in the proof of Lemma~\ref{lem:lemma4}, we integrate by parts and get 
\begin{align}
  c_1 &= -\frac{1}{\pi}\left(D(m_0) \int_B \Delta m_1\, dx - K_0m_0 \int_B \Delta \sigma_1 \, dx\right) \\
  &=-\frac{1}{\pi} \left(D(m_0) \int_{\partial B} m_{1,\rho}(1, \theta)\,d \theta - K_0 m_0 \int_{\partial B} \sigma_{1,\rho}(1,\theta)\,d\theta\right) \\
  &= -\frac{1}{\pi}\left(-K_0m_0 \int_{\partial B} \frac{R_0}{K_0}V_1 \cos \theta\,d \theta\right) =0
.\end{align}

Now, to solve $m_1$ in terms of $\sigma_1$, introduce a new variable 
\[
u_1 := D(m_0) m_1 - K_0 m_0 \sigma_1.
\] 
From~\eqref{eq:zero_ev_a},\eqref{eq:zero_ev_c},\eqref{eq:zero_ev_f}, $u_1$ solves
\begin{subequations}\label{eq:u_prob}
  \begin{numcases}{}
    \Delta_{(\rho,\theta)} u_1 =0 \\
    u_{1,\rho}(1,\theta) =-m_0R_0V_1 \cos \theta
  \end{numcases}
\end{subequations}
Then, the solution to~\eqref{eq:u_prob} is given by
\begin{equation}
  u_1 = -m_0R_0V_1 \rho \cos \theta + C,\ C \in \R.
\end{equation}
This allows to solve $m_1$ in terms of $\sigma_1$ and obtain the following problem for $\sigma_1, R_1, V_1, C$

\begin{subequations}\label{eq:sRVC}
  \begin{numcases}{}
    \frac{Z}{R_0^2} \Delta_{(\rho, \theta)} \sigma_1 +\left(\frac{K_0Pm_0}{D(m_0)}-1\right) \sigma_1 = \frac{P m_0 R_0}{D(m_0)}V_1 \rho \cos \theta +C \\
 \sigma_1(1,\theta)  +R_0 \int_{0}^{2\pi}R_1(\theta)\, d \theta-\gamma \dfrac{R_1(\theta)+R_1''(\theta)}{R_0^2} =0 \\
    R_{0}^{2} \int_B \left(\frac{K_0 m_0}{D(m_0)} \sigma_1 + \frac{C}{D(m_0)}\right) \, dx+  \int_{0}^{2\pi} R_1(\theta)\,d\theta  =0 \\
    \dfrac{K_0}{R_0} \sigma_{1,\rho}\left(1, \theta\right) =  V_1 \cos \theta.
  \end{numcases}
\end{subequations}
We look for a symmetrical around x-axis solutions of \eqref{eq:sRVC} in the form of Fourier series
\begin{align}
    \sigma_1(\rho, \theta) &= S_0(\rho) + \sum_{n=1}^\infty S_n(\rho) \cos(n \theta),\\
    R_1(\theta) & = R_{1,0} + \sum_{n=2}^\infty R_{1,n} \cos(n \theta),
\end{align}
where all coefficients $S_n(\rho)$ ($n\geq 1$) vanish at zero and $S_0'(0) =0$. 
Note that because we excluded the shifts,
there is no $\cos \theta$ term in the expansion of $R_1$.

For $n=0$ we get the system for $S_0, R_{1,0}, C$
\begin{subequations}
\begin{numcases}{}
\frac{Z}{R_0^2} \Big(S_0'' {+}\frac{1}{\rho}S_0'\Big) + \left(\frac{K_0Pm_0}{D(m_0)}-1\right) S_0 =C, \quad 0< \rho < 1 \label{eq:S0a}\\
 S_0'(1) =0 \label{eq:S0b}\\
 S_0'(0) =0\label{eq:S0c}\\
S_0(1) + 2\pi R_0 R_{1,0} - \frac{\gamma}{R_0^2} R_{1,0} =0\label{eq:S0d}\\ 
2 \pi R_0^2 \frac{K_0m_0}{D(m_0)} \int_0^1 S_0(\rho) \rho\,d\rho + \pi \frac{R_0^2}{D(m_0)} C +  2\pi R_{1,0} =0 \label{eq:S0e}
\end{numcases}
\end{subequations}
From~\eqref{eq:S0a}, \eqref{eq:S0b}, \eqref{eq:S0c} we obtain $S_0=\left(\frac{K_0Pm_0}{D(m_0)}-1\right)^{-1}C$, which after substituting in~\eqref{eq:S0d} allows one to find $R_{1,0} =-\left(2\pi R_0 - \frac{\gamma}{R_0^2}\right)^{-1}\left(\frac{K_0Pm_0}{D(m_0)}-1\right)^{-1}C$. However, substituting this into~\eqref{eq:S0e}, we get that $C$ must be equal to zero, and therefore $S_0=R_{1,0}=0.$

For each $n \geq 2$, we obtain
\begin{subequations}\label{eq:cos_n}
\begin{numcases}{}
\frac{Z}{R_0^2} \Big(S_n'' {+}\frac{1}{\rho}S_n' {-} \frac{n^2}{\rho^2} S_n\Big)  + \left(\frac{K_0Pm_0}{D(m_0)}-1\right) S_n =0, \quad 0< \rho < 1 \label{eq:cos_n_1}\\
 S_n'(1) =0 \\
 { S_n(0) = 0 } \label{eq:cos_n_5} \\
 S_n(1) - \gamma \frac{(1-n^2)R_{1,n}}{R_0^2} =0 \label{eq:cos_n_R}
\end{numcases}
\end{subequations}
{For $n\geq 2$, the solution to system~\eqref{eq:cos_n_1},\eqref{eq:cos_n_5} is given by $S_n(\rho) = C J_n(\alpha \rho)$, and using condition~\eqref{eq:non-resonance}, we obtain $C=0$. From which, using~\eqref{eq:cos_n_R} it follows that $\cos(n\theta)$ modes admit only the constant zero solution
\begin{equation*}
    R_n= S_n= 0, n \geq 2.
\end{equation*}
}
Therefore, any zero eigenfunction has the form
\begin{equation}\label{eq:simple_ev}
   \left(m_1, R_1, V_1, c_1\right)=\left(\left(\frac{K_0m_0}{D(m_0)}S_1(\rho)-\frac{m_0R_0}{D(m_0)}V_1 \rho\right)\cos\theta, 0, V_1, 0\right),
\end{equation}
where $S_1(\rho)$ solves
\begin{subequations}\label{eq:32}
    \begin{numcases}{}
        \frac{Z}{R_0^2} (S_1'' +\frac{1}{\rho}S_1' - \frac{1}{\rho^2} S_1) + \left(\dfrac{P K_0 m_0}{D(m_0)} -1 \right)S_1 = \dfrac{P R_0 m_0}{D(m_0)}V_1\rho, \quad 0< \rho < 1\label{eq:32a}\\
        \frac{K_0}{R_0}S_1'(1) = V_1 \\
        { S_1(0) = 0}\label{eq:32c} \\
        S_1(1) = 0 \label{eq:32d}
    \end{numcases}
\end{subequations}
Now there is a unique solution to \eqref{eq:32a}-\eqref{eq:32c}, which is given by
\begin{equation}
S_{1}(\rho)=
\frac{R_{0}}{P\,K_{0}\,m_{0}-D(m_{0})}
\left(
P\,m_{0}\,\rho
-\frac{D(m_{0})}{K_{0}}\,
\frac{J_{1}\bigl(\alpha\rho\bigr)}{\alpha\,J_{1}'(\alpha)}
\right)V_1,
\label{eq:S1_solution}
\end{equation}
where $\alpha$ is given by
\begin{equation}
\alpha=\frac{R_{0}}{\sqrt{Z}}\,
\sqrt{\frac{P\,K_{0}\,m_{0}}{D(m_{0})}-1}
\end{equation}
and  $J_1$ is the Bessel function of first kind. This solution also satisfies \eqref{eq:32d} provided that $K_0$ satisfies the transcendental equation~\eqref{eq:Bif_cond}.

Thus, we have shown that the eigenvalue $0$ is simple and that the eigenvector corresponding to it is given by~\eqref{eq:simple_ev} with $S_1(\rho)$ given by \eqref{eq:S1_solution}.

{\bf Step 3. }{\it Fredholm and zero-index property in~\ref{item:CR3}}
 {
In Step 2 we showed that $\mathrm{dim}(\mathrm{Ker}(F_x(x_0, K_0))) = 1$. 
Now we will prove that our operator has index zero.
To this end, first consider the $\sigma-$operator. A solution operator $S_x$ to problem~\eqref{eq:S_op} from the classic elliptic theory is a compact operator from $X_m \times \R \times X_R \to H^2_{sym}(B)$. 
Now, a Laplacian is a bounded operator from $H^2_{sym}(B) \to L^2_{sym}(B)$, and therefore $\Delta \sigma(m, V, R)$ is a compact operator from $X_m \times \R \times X_R$ to $L^2_{sym}(B)$ as a composition of bounded and compact operators.
 Moreover, the trace of the $H^2(B)$ function is compact operator into $L^2(\partial B)$. In the light of the reasoning above, we can rewrite $F_x(x_0, K_0)$ as 
\begin{equation}
    F_x\left(x_0, K_0 \right)
    = 
   \begin{pmatrix} 
    D(m_0)\dfrac{1}{R_0^2}\Delta_{(\rho,\theta)} m +c \\  
  \dfrac{1}{R_0}\hat m_\rho\left(1, \theta\right)\\
  R_0 \int_{0}^{2\pi}R(\theta)\, d \theta-\gamma \dfrac{R(\theta)+R''(\theta)}{R_0^2}\\
R_{0}^{2}\int_B m\, dx + m_0R_0\int_{0}^{2\pi} R(\theta)\,d\theta 
 \end{pmatrix} + \begin{pmatrix}-  \dfrac{K_0 m_0}{R_0^2}\Delta_{(\rho,\theta)} \sigma \\
0 \\
\sigma\left(1, \theta \right) \\
0
  \end{pmatrix} =: T+K
  ,\end{equation}
where the second term $K$ is a compact operator from $X$ to $L^2_{sym}(B) \times L^2_{sym}(\partial B) \times  L^2_{sym}(\partial B) \times \R$. Therefore, as the compact perturbation does not change the index of an operator,
\[
\mathrm{ind}(F_x(x_0, K_0)) = \mathrm{ind}\left(T\right).
\]
Now, it is easy to compute directly both the $\mathrm{ker}(T)$ and $\mathrm{ran}(T)$. 
Indeed, to find the kernel, we need to solve the following system for $m, R, V, c$
\begin{subequations}
  \begin{numcases}{}
    D(m_0)\dfrac{1}{R_0^2}\Delta_{(\rho,\theta)} m +c = 0, \label{eq:70a}\\
     m_\rho\left(1, \theta\right) =0, \label{eq:70b}\\
     R_0 \int_{0}^{2\pi}R(\theta)\, d \theta-\gamma \dfrac{R(\theta)+R''(\theta)}{R_0^2} =0, \label{eq:70c}\\
     R_{0}^{2}\int_B m\, dx + m_0R_0\int_{0}^{2\pi} R(\theta)\,d\theta =0. \label{eq:70d}
  \end{numcases}
\end{subequations}
First, note that from integrating~\eqref{eq:70a} over the ball, we see that $c=0$. From~\eqref{eq:70c},\eqref{eq:70d} using Fourier analysis it is easy to see that $R=0$, and therefore from~\eqref{eq:70d} $\int_B m\, dx =0$. Using~\eqref{eq:70a},\eqref{eq:70b} we get that $m$ must be a constant and in the light of $\int_B m\, dx=0$, it must be zero as well. Thus, the kernel is one-dimensional and only consists of $\left(0,0,V,0\right), V \in \R$.

To find the $\mathrm{ran}(T)$, we find the necessary and sufficient conditions for an existence of a solution $\left(m, R, V,c\right)$ to 
\begin{subequations}
  \begin{numcases}{}
    D(m_0)\dfrac{1}{R_0^2}\Delta_{(\rho,\theta)} m +c = f(\rho,\theta), \label{eq:71a}\\
     m_\rho\left(1, \theta\right) =g(\theta), \label{eq:71b}\\
     R_0 \int_{0}^{2\pi}R(\theta)\, d \theta-\gamma \dfrac{R(\theta)+R''(\theta)}{R_0^2} =a(\theta), \label{eq:71c}\\
     R_{0}^{2}\int_B m\, dx + m_0R_0\int_{0}^{2\pi} R(\theta)\,d\theta =b. \label{eq:71d}
  \end{numcases}
\end{subequations}
First, one can find $R\in X_R$ using Fourier series from~\eqref{eq:71c}, provided that \[\int_0^{2\pi} a(\theta) \cos\theta\, d\theta=0.\] Now from compatibility conditions between~\eqref{eq:71a} and~\eqref{eq:71b} we must take 
\[c=\int_{\partial B} g\,dS-\int_B f\,dx.\]
Then, there is a unique, up to an additive constant, solution $m$ of the BVP~\eqref{eq:71a},\eqref{eq:71b} and this constant can be found using the integral condition~\eqref{eq:71d}. Therefore, the $\mathrm{codim}\left(\mathrm{ran}(T)\right)=1$. 

Therefore, the index of the operator $F_x\left(x_0, K_0\right)$ is equal to the index of $T$ and is equal to zero. 
 }

{\bf Step 4.} {\it Establishing the transversality condition~\ref{item:CR4}.}  
We claim that our non-degeneracy condition on the physical parameters \eqref{eq:tranversality_cond}
implies the transversality condition~\ref{item:CR4}.

To this end, we compute the second derivative of the operator $F$ defined in \eqref{eq:F_nonlinear}:
\begin{equation}  
F_{K,x}(x_0, K_0) \begin{pmatrix}
    m \\ 
    R \\
    V \\
    c
\end{pmatrix} = \begin{pmatrix}
    -\frac{m_0}{R_0^2}\Delta \sigma  - \frac{K_0 m_0}{R_0^2}\Delta \dot \sigma\\
    0 \\
    \dot\sigma(1,\theta) \\
    0 \\
\end{pmatrix},
\end{equation}
where $\dot \sigma$ is a solution to
\begin{subequations}
\begin{numcases}{}
\dfrac{Z}{R_0^2}\Delta\dot\sigma=\dot\sigma \quad \text{in }B,\\
\dfrac{K_0}{R_0}\,\partial_\rho\dot\sigma(1,\theta)=-\dfrac{1}{R_0}\,\partial_\rho\sigma(1,\theta).
\end{numcases}
\end{subequations}
We then directly show that $F_{K,x}(x_0, K_0) x_1 \not\in \mathrm{Range}(F_x(x_0,K_0))$ by proving that there are no solutions $x \in X$ to the linear system
\begin{subequations}\label{eq:41}
    \begin{numcases}{}
\dfrac{Z}{R_{0}^2} \Delta_{(\rho,\theta)}\sigma - \sigma + P \hat m =0 \label{eq:41_a}\\
    D(m_0)\dfrac{1}{R_0^2}\Delta_{(\rho,\theta)} m - K_0 m_0 \dfrac{1}{R_0^2} \Delta_{(\rho,\theta)} \sigma  
    =-\frac{m_0}{R_0^2} \Delta_{(\rho,\theta)} \sigma_1 - \frac{K_0 m_0}{R_0^2}\Delta \dot \sigma_1\\
  \dfrac{1}{R_0}\hat m_\rho\left(1, \theta\right) =0\\
  \sigma\left(1, \theta \right) +R_0 \int_{0}^{2\pi}R(\theta)\, d \theta-\gamma \dfrac{R(\theta)+R''(\theta)}{R_0^2} =\dot\sigma_1(1,\theta)\\
 \dfrac{K_0}{R_0} \sigma_\rho\left(1, \theta\right) =   V \cos \theta\label{eq:41_e}   .   
    \end{numcases}
\end{subequations}
Again, we utilize Fourier analysis and as the only non-zero mode of $\sigma_1$, and therefore $\dot \sigma_1$  is the $\cos \theta$ mode, we repeat the argument from Step 2 and see that if the solution to \eqref{eq:41} exists it must be of the form
\begin{equation}
    \sigma(\rho, \theta) = \hat S_1(\rho) \cos \theta, m(\rho, \theta) = \hat M_1(\rho) \cos \theta
\end{equation}
Therefore, system \eqref{eq:41} is equivalent to the system for $\hat S_1(\rho), \hat M_1(\rho)$
\begin{subequations}
    \begin{numcases}{}
        \frac{Z}{R_0^2} \Big(\hat S_1'' +\frac{1}{\rho}\hat S_1' - \frac{1}{\rho^2} \hat S_1\Big) -\hat  S_1 + P \hat M_1 =0, \hfill 0< \rho < 1\label{eq:44_a}\\
\dfrac{ D(m_0)}{R_0^2}\Big(\hat M_1'' {+} \frac{1}{\rho}\hat M_1' {-} \frac{1}{\rho^2} \hat M_1\Big) - \dfrac{ K_0 m_0}{R_0^2}\Big(\hat S_1'' {+} \frac{1}{\rho}\hat S_1' {-} \frac{1}{\rho^2} \hat S_1\Big)\nonumber \\
 =  \dfrac{m_0}{R_0^2}\Big( S_1'' {+} \frac{1}{\rho} S_1' {-} \frac{1}{\rho^2} S_1\Big) - \dfrac{K_0m_0}{R_0^2} \Big( \dot S_1'' {+} \frac{1}{\rho} \dot S_1' {-} \frac{1}{\rho^2} \dot S_1\Big), \hfill 0 < \rho < 1 \label{eq:44_b} \\
 \hat M_1'(1) =0\label{eq:44_c} \\
\hat  S_1 (1) = \dot S_1(1) \\
\frac{K_0}{R_0}\hat S_1'(1) = V \\
 { \hat M_1(0) =0 }\\
 {\hat S_1(0) = 0}.\label{eq:44_g}
    \end{numcases}
\end{subequations}

Rearranging terms in \eqref{eq:44_a}-\eqref{eq:44_g}, changing variables according to

\[
\tilde S_1(\rho)= K_0\hat S_1(\rho) - S_1(\rho) - K_0 \dot S_1(\rho),
\]

and using \eqref{eq:32c},\eqref{eq:32d} the system becomes
\begin{subequations}
    \begin{numcases}{}
        \frac{Z}{R_0^2} \Big(\tilde S_1'' {+}\frac{1}{\rho}\tilde S_1' {-} \frac{1}{\rho^2} \tilde S_1\Big) -\tilde  S_1 + P K_0 \hat M_1 = \nonumber \\
        \quad \quad =\frac{Z}{R_0^2} \Big(S_1'' {+}\frac{1}{\rho} S_1' {-} \frac{1}{\rho^2} S_1\Big) -  S_1, \hfill\quad 0{<} \rho {<} 1\label{eq:45_a}\\
 D(m_0)\dfrac{1}{R_0^2}\Big(\hat M_1'' {+} \frac{1}{\rho}\hat M_1' {-} \frac{1}{\rho^2} \hat M_1\Big) - m_0 \dfrac{1}{R_0^2}\Big(\tilde S_1'' {+} \frac{1}{\rho}\tilde S_1' {-} \frac{1}{\rho^2} \tilde S_1\Big) = 0, \hfill \quad 0 {<} \rho {<} 1 \label{eq:45_b} \\
 \hat M_1'(1) =0\label{eq:45_c} \\
\tilde  S_1 (1) = 0 \label{eq:45_d}\\
 \tilde S_1'(1) = R_0 V \label{eq:45_e}\\
 { \hat M_1(0) =0 }\label{eq:45_f}\\
 {\tilde S_1(0) = 0.}\label{eq:45_g}
    \end{numcases}
\end{subequations}
Using \eqref{eq:45_b}, we can express $\hat M_1$ in terms of $\tilde S_1$:
\begin{equation}\label{eq:hat_m_in_s}
    \hat M_1 = \frac{m_0}{D(m_0)} \tilde S_1 + C_1 \rho + C_2 \frac1\rho;
\end{equation}
and after substituting in boundary conditions \eqref{eq:45_f}, \eqref{eq:45_g} it follows $C_2=0$. From boundary conditions \eqref{eq:45_c}, \eqref{eq:45_e} it follows that $C_1 = -R_0 V m_0 / D(m_0)$. Substituting \eqref{eq:hat_m_in_s} in the system \eqref{eq:45_a}-\eqref{eq:45_g}, it reduces to 
\begin{subequations}
    \begin{numcases}{}
         \frac{Z}{R_0^2} (\tilde S_1'' +\frac{1}{\rho}\tilde S_1' - \frac{1}{\rho^2} \tilde S_1) + \left(\frac{P K_0 m_0}{D(m_0)}-1\right) \tilde S_1  \nonumber\\ = \frac{Z}{R_0^2} (S_1'' +\frac{1}{\rho} S_1' - \frac{1}{\rho^2} S_1) -  S_1 + \frac{PK_0 R_0 m_0V}{D(m_0)}\rho , \hfill 0< \rho < 1 \label{eq:47a}\\
         \tilde S_1(0) =0 \\
         \tilde S_1(1) =0  \\
         \tilde S_1'(1) = R_0V.
    \end{numcases}
\end{subequations}
Now, we can simplify the right-hand side of \eqref{eq:47a} using the PDE \eqref{eq:32a} and the solution formula \eqref{eq:S1_solution}
\begin{align*}
&\frac{Z}{R_0^2} (S_1'' +\frac{1}{\rho} S_1' - \frac{1}{\rho^2} S_1) -  S_1 + \frac{PK_0 R_0 m_0V}{D(m_0)}\rho 
\\&=  \frac{PR_0m_0}{D(m_0)} \rho -\frac{P K_0 m_0}{D(m_0)} S_1 + \frac{PK_0 R_0 m_0V}{D(m_0)} \rho 
\\& =\left(\frac{PR_0m_0}{D(m_0)} + \frac{PK_0 R_0 m_0V}{D(m_0)} - \frac{P^2K_0m_0^2R_0}{D(m_0)(PK_0m_0-D(m_0))}\right)\rho 
\\&\qquad+ \frac{PR_0m_0}{(PK_0m_0 - D(m_0))\alpha J_1'(\alpha)}{J_1(\alpha \rho)}
\\&=:A(V) \rho + B J_1(\alpha \rho).
\end{align*}

Finally, using \eqref{eq:32a} the right-hand side of \eqref{eq:47a} can be simplified and we derive the final version of the ODE system for $(\tilde S_1(\rho), V)$
\begin{subequations}
    \begin{numcases}{}
         \frac{Z}{R_0^2} \Big(\tilde S_1'' {+}\frac{1}{\rho}\tilde S_1' {-} \frac{1}{\rho^2} \tilde S_1\Big) + \left(\frac{P K_0 m_0}{D(m_0)}-1\right) \tilde S_1  =A(V)\rho + B {J_1(\alpha \rho)}\label{eq:48a}\\
         \tilde S_1(0) =0 \label{eq:48b} \\
         \tilde S_1(1) =0 \label{eq:48c} \\
         \tilde S_1'(1) = R_0V. \label{eq:48d}
    \end{numcases}
\end{subequations}

For any $V\in \R$, the solution to \eqref{eq:48a}--\eqref{eq:48c} is given by
\begin{align}\label{eq:62ab_soln}
    \tilde S_1(\rho) = &\frac{A(V)}{\frac{PK_0m_0}{D(m_0)}-1}\rho -\frac{A(V)}{\frac{PK_0m_0}{D(m_0)}-1}\frac{J_1(\alpha \rho)}{ J_1(\alpha)}  \nonumber\\
    +&\frac{\pi}{2}\frac{R_0^2}{Z}B \left(- J_1(\alpha \rho)\int_0^\rho s Y_1(\alpha s) J_1(\alpha s) ds +Y_1(\alpha \rho) \int_0^\rho s J_1(\alpha s) J_1(\alpha s)ds\right) \nonumber \\
    -&\frac{\pi}{2}\frac{R_0^2}{Z}B \left(- J_1(\alpha)\int_0^1 s Y_1(\alpha s) J_1(\alpha s) ds +Y_1(\alpha) \int_0^1 s J_1(\alpha s)^2ds\right) \frac{J_1(\alpha \rho)}{ J_1(\alpha)}.
\end{align}
Now we must show that  \eqref{eq:62ab_soln} does not satisfy the extra boundary condition \eqref{eq:48d}. Computing the derivative of $\tilde S_1(\rho)$ at $\rho=1$ and simplifying the terms yields
\begin{align}\label{eq:s1p}
    \tilde S_1'(1) 
    &=R_0V + \frac{D(m_0) R_0}{K_0(PK_0m_0-D(m_0))}\left(-1 +\frac{R_0^2}{Z} \frac{1}{\alpha J_1'(\alpha)} \int_0^1 s J_1(\alpha s)^2\,ds\right)
\end{align}
which -- thanks to our non-degeneracy condition~\eqref{eq:tranversality_cond} -- contradicts~\eqref{eq:48d}.
This contradiction proves the transversality condition~\ref{item:CR4}.

Thus we have verified all conditions~\ref{item:CR1}--\ref{item:CR4} of the Crandall-Rabinowitz theorem and the Proposition~\ref{thm:bif_and_tw_existence} is proven. 
\end{proof}

\section{Change of coordinates to fixed boundary}\label{app:free_to_fixed}

\subsection{Reformulation of the problem on the unit disk}
The original system of the equations \eqref{eq:1_2d}-\eqref{eq:5_2d} is posed on the moving domain $\Omega(t)$ with the free boundary. 
To develop the proper functional setting for the Crandall-Rabinowitz theorem we map the problem to a fixed domain. We do it in two steps. First, we move the domain such that the center of mass of the cell is fixed. After that we can parametrize the boundary in polar coordinates and use the so-called Hanzawa transform to map the problem to the unit disk.

The center of mass of the domain $\Omega(t)$ is given by 
\begin{equation}
    \mathbf c(t) = \frac1{\lvert \Omega(t) \rvert }\int_{\Omega(t)} x \, dx.
\end{equation}
Without loss of generality, we can assume that the traveling wave solution moves in the direction of the $x_1$-axis and from the symmetry of the cell we can assume that the center of the cell is always on the $x_1$-axis. To this end, we can rewrite
\begin{equation}
  \mathbf c(t) = c_1(t)  \mathbf e_{1} = \dfrac{\int_{\Omega(t)}x_1 dx}{|\Omega(t)|}\mathbf e_{1}.
\end{equation}

After the shift of the coordinates as well as the domain

\begin{equation}\label{eq:shift_of_coord}
\begin{aligned}
x &\mapsto x- \mathbf c(t),\\
\tilde\Omega(t) &:= \Omega(t) - \mathbf{c}(t)
           = \{\, \mathbf{y}-\mathbf{c}(t) : \mathbf{y}\in\Omega(t) \,\},
\end{aligned}
\end{equation}
the system \eqref{eq:1_2d}--\eqref{eq:5_2d} becomes

\begin{align}
    \label{eq:1_2d_cm}
    Z\Delta \sigma &= \sigma - Pm, \quad  &&(r, \theta) \in \tilde \Omega(t)\\
    \label{eq:2_2d_cm}
    \partial_t m - \frac{dc_1}{dt} \mathbf e_{1} \cdot \nabla m &= \text{div}\,(D(m) \nabla m - K m \nabla \sigma), \quad &&(r, \theta) \in \tilde \Omega(t)\\
    \label{eq:3_2d_cm}
    \partial_\nu m(x) &= 0, \quad &&(r, \theta) \in \partial\tilde \Omega(t)\\
    \label{eq:4_2d_cm}
    \sigma(x) &= 1-|\Omega(t)|- \gamma H, \quad  &&(r, \theta) \in \partial\tilde\Omega(t)\\
    \label{eq:5_2d_cm}
    K \partial_\nu \sigma (x) &=(V + \frac{dc_1}{dt}  \mathbf e_1) \cdot \nu, \quad&& x \in \partial \tilde\Omega(t).
\end{align}

Now we can parametrize {in polar coordinates} the boundary {\[\partial \tilde \Omega(t) = \{(R(\theta, t),\theta), \theta \in [0, 2\pi]\},\]}
where $R(\theta, t)$ is the radial distance from the center (now fixed at $0$) to the boundary at angle $\theta$. $V = \partial_t R(\theta, t)$ represents the radial velocity in \eqref{eq:5_2d_cm}. We compute $\frac{dc_1}{dt}$ in new coordinates:
\begin{align*}
    \frac{dc_1}{dt} 
    &= \frac{d}{dt} \frac{1}{|\Omega(t)|} \int_{\Omega(t)} x_1\,dx 
    = \frac{1}{|\Omega(t)|}\int_{\partial \Omega(t)} x_1 V_\nu \,ds - \frac{1}{|\Omega(t)|^2}\int_{\Omega(t)}x_1\,dx \int_{\partial \Omega(t)}V_\nu\, ds 
    \\&= 
    \frac{1}{|\Omega(t)|} \int_{\partial \Omega(t)} (x_1-c) V_{\nu}\, ds = \frac{1}{|\tilde \Omega(t)|} \int_{\partial \tilde \Omega(t)} x_1 V_\nu\, ds = \frac{K}{|\tilde \Omega(t)|} \int_{\partial \tilde \Omega(t)} x_1 \partial_\nu \sigma\, ds 
    \\&= 
    \frac{K}{|\tilde \Omega(t)|} \int_0^{2\pi} R(\theta,t)\cos \theta \partial_{\nu} {\sigma} \sqrt{ R(\theta,t)^2 + R_{\theta}(\theta,t)^2 }\, d\theta.
\end{align*}
This allows to derive the following equation for $\partial_t R(\theta, t)$
\begin{equation}
    \frac{R(\theta, t)}{\sqrt{R_\theta(\theta,t)^2 + R(\theta,t)^2}}\partial_t R(\theta, t) =  K \partial_\nu \sigma - \frac{dc_1}{dt} \mathbf e_1 \cdot \nu 
\end{equation}

Finally, we map our problem to the unit ball via the Hanzawa transform
\begin{equation}\label{eq:Hansawa}
    r(\rho, \theta, t) = R_0 \rho + \chi(\rho)(R(\theta,t)-R_0),\ \rho \in [0,1],
\end{equation}
where $\chi \in C^\infty[0,1]$ is monotone increasing from 0 to 1. We also assume that $\chi=0$ for $\rho<1/3$, and $\chi=1$ for $\rho>2/3$. For a function $u(r, \theta,t),\ 0\leq r\leq R(\theta,t)$ after change of coordinates we consider a function $v(\rho, \theta,t) = u(r(\rho, \theta,t), \theta,t),\ 0 \leq \rho \leq 1$.
A direct computation allows us to find how the derivatives in the new coordinates $(\rho, \theta,t)$:
\begin{equation}\label{eq:fc_pd}
    v_\rho = u_r r_\rho, \quad v_\theta = u_r r_\theta+ u_\theta, \quad v_t = u_r r_t + u_t.
\end{equation}
The gradient, divergence, and the Laplacian in new coordinates are given by
\begin{align}
\label{eq:tilde_grad}
   \tilde{\nabla} v &= \nabla u = 
   u_{r} e_{r} + \frac{1}{r}u_{\theta} e_{\theta} = \frac{1}{r_{\rho}}v_{\rho}e_{r} + \frac{1}{r}(-\frac{r_{\theta}}{r_{\rho}} v_{\rho}+v_{\theta})e_{\theta},
   \\
\tilde{\text{div\,}} \mathbf F &=
\frac{1}{r} (r\mathbf F^{r})_{r} + \frac{1}{r} (\mathbf F^{\theta})_{\theta} = 
\frac{1}{r r_{\rho}}(r \mathbf F^r)_{\rho} + \frac{1}{r}\left( -\frac{r_{\theta}}{r_{\rho}}\mathbf F^\theta_{\rho} + \mathbf{F}^\theta_{\theta} \right),
\\
\label{eq:fc_laplace}
    \tilde \Delta v &= {\Delta}u = \frac{1}{r_{\rho}^2}\left( 1+ \frac{r_{\theta}^2}{r^2} \right) u_{\rho \rho} - \frac{2r_{\theta}}{r_{\rho}r^2} u_{\rho \theta} + \frac{1}{r^2} u_{\theta \theta} \nonumber\\
    &+\frac{1}{r_{\rho}}\left( -\frac{r_{\rho \rho}}{r_{\rho}^2}\left( 1+\frac{r_\theta^2}{r^2} \right) + \frac{2r_{\theta}}{r_{\rho}} \frac{r_{\rho \theta}}{r^2} - \frac{r_{\theta \theta}}{r^2} + \frac{1}{r} \right) u_{\rho}.
\end{align}

Note, that after the change of coordinates, the volume of the new domain is constant, but we still can make sense of the term $|\Omega(t)|$ defining it via 
\begin{equation}
    |\Omega(t)| = \frac{1}{2} \int_0^{2\pi} R(\theta,t)^2\, d\theta.
\end{equation}
{The condition \eqref{eq:myosin_mass} for the total myosin rewrites as

\begin{align}
  \int_{\Omega(t)} m(x,t)\,dx &= \int_0^{2\pi} \int_0^{R(\theta, t)} m(r, 
  \theta, t) r\,drd\theta \\ 
  &= \int_0^{2\pi} \int_0^1 m(r(\rho), \theta, t) r(\rho, \theta, t) r_\rho\,d\rho d\theta.
\end{align}
Denote 
\begin{equation}\label{eq:int_weight}
  w[R](\rho, \theta, t) := \dfrac{r(\rho, \theta, t) r_\rho(\rho, \theta, t)}{\rho}.
\end{equation}
}
Finally, after the change of coordinates the normal vector, normal derivative at the boundary, and the curvature at the boundary are given by
\begin{equation}
    \nu[R] = \left( \frac{R}{\sqrt{R_\theta^2+ R^2}}, \frac{-R_\theta}{\sqrt{R_\theta^2+R^2}}\right),
\end{equation}
\begin{equation}
   N[R] u = \frac{1}{R} \left( 1+\frac{R_\theta^2}{R^2} \right)^{1/2} u_{\rho}(1, \theta) - \frac{1}{R} \frac{R_\theta/ R}{\sqrt{ 1+ \frac{R_\theta^2}{R^2}}} u_{\theta} (1, \theta),
\end{equation}
\begin{equation}\label{eq:h[r]}
    H[R] = \frac{-\frac{R_{\theta \theta}}{R}+2\frac{R_\theta^2}{R^2}+1}{R\sqrt{\left( 1+\frac{R_\theta^2}{R^2} \right)^3}}.
\end{equation}

Therefore, after the change of coordinates given by \eqref{eq:Hansawa} in the light of \eqref{eq:fc_pd}-\eqref{eq:h[r]} the system \eqref{eq:1_2d_cm}-\eqref{eq:5_2d_cm} {becomes the following coupled nonlinear PDE on a fixed unit ball $B(0,1)$ domain}

\begin{align}{}
    \label{eq:1_2d_fixed}
    Z\tilde \Delta \sigma &= \sigma - Pm,
    \\
    \label{eq:2_2d_fixed}
    \partial_t m + {m_r \chi(\rho) \partial_t R(\theta, t)} - \frac{dc_1}{dt} \mathbf e_{1} \cdot \tilde\nabla m &= \tilde{\text{div}}\,(D(m) \tilde \nabla m - K m \tilde \nabla \sigma),
\end{align}
with boundary conditions on $\partial B(0,1):$
\begin{align}
    \label{eq:3_2d_fixed}
    N[R] m &= 0,
    \\
    \label{eq:4_2d_fixed}
    \sigma(x) &= 1-|\Omega(t)|- \gamma H[R],
    \\
    \label{eq:5_2d_fixed}
   \frac{R(\theta, t)}{\sqrt{R_\theta(\theta,t)^2 + R(\theta,t)^2}}\partial_t R(\theta, t) &=  K N[R]\sigma - \frac{dc}{dt} \mathbf e_1 \cdot \nu[R],
\end{align}

{Now, in this new system of coordinates, the traveling wave solutions are time independent and solve
\begin{align}
    \label{app:eq:1_2d_fixed_TW}
    Z\tilde \Delta \sigma &= \sigma - Pm, \quad  && (r, \theta) \in B(0,1)\\
    \label{app:eq:2_2d_fixed_TW}
    -V \mathbf e_{1} \cdot \tilde\nabla m &= \tilde{\text{div}}\,(D(m) \tilde \nabla m - K m \tilde \nabla \sigma), \quad &&(r, \theta) \in B(0,1)\\
    \label{app:eq:3_2d_fixed_TW}
    N[R] m &= 0, \quad &&(r, \theta) \in \p B(0,1)\\
    \label{app:eq:4_2d_fixed_TW}
    \sigma(x) &= 1-|\Omega(t)|- \gamma H[R], \quad  &&(r, \theta) \in \p B(0,1)\\
    \label{app:eq:5_2d_fixed_TW}
    0&=  K N[R]\sigma - V \mathbf e_{1} \cdot \nu[R], \quad  &&(r, \theta) \in \p B(0,1)
\end{align}
with an integral constraint

\begin{equation}\label{app:eq:integral_const}
  \int_B m w[R]\, dx=1.
\end{equation}
}

\subsection{Functional analytic setup for Crandall-Rabinowitz theorem}

Now the solutions of  \eqref{eq:1_2d_fixed_TW}-\eqref{eq:integral_const} can be found as solutions to $F(x, K) =0$, where $F(x,K)$ is given by \eqref{eq:F_nonlinear} and it maps from the input Banach space
\begin{equation}\label{eq:Domain_space}
X = X_m \times X_R \times \R \times \R
\end{equation}
where  
\begin{align}
   X_m &:= H^2_{sym}(B) = \{m \in H^2(B) : m(\rho, \theta) = m(\rho, -\theta)\}, \label{eq:X_m}\\
   X_R &:= \{R \in H^{7/2}(0, 2\pi) : R\text{ is even, $2\pi$-periodic},  \int_0^{2\pi} R(\theta) \cos \theta d\theta =0 \} \label{eq:X_R},
\end{align}
 to the output space $Y = (y_1, y_2, y_3, y_4)$

\begin{align}\label{eq:target_space}
&Y=
\Big\{ (y_1,\ldots, y_4) \in L^2_{sym}(B) \times
L^2_{sym}(\partial B) \times 
L^2_{sym}(\partial B) \times 
\R
\},
\end{align}

where the $sym$ spaces are understood in the same sense as in~\eqref{eq:X_m}.

{The constraint $\int_0^{2\pi} R(\theta)\cos\theta\, d\theta =0$ eliminates horizontal shifts of the cell from consideration.

We present below the full formulation of the Crandall-Rabinowitz theorem (see Theorem 1.7 in ~\cite{crandall1971bifurcation}).

\begin{theorem}
Let $X,Y$ be Banach spaces and let $F(x,\mu)$ be a map from a neighborhood of $(0,\mu_{0})$ in $X\times\mathbb{R}$ into $Y$. Suppose that
\begin{enumerate}
  \item[(i)] $F(0,\mu)=0$ for all $\mu$ in a neighborhood of $\mu_{0}$,
  \item[(ii)] $F(x, \mu)$ is a $C^p$ map in a neighborhood of $(0,\mu_{0})$,
  \item[(iii)]  $\ker F_{x}(0,\mu_{0})$ is one-dimensional, spanned by $x_{1}$, $\operatorname{codim}(\operatorname{Ran} F_{x}(0,\mu_{0}))=1$,
  \item[(iv)] $F_{\mu x}(0,\mu_{0})\,x_{1}\notin \operatorname{Ran} F_{x}(0,\mu_{0})$.
\end{enumerate}
Then $(0,\mu_{0})$ is a bifurcation point of the equation $F(x,\mu)=0$ in the following sense: in a neighborhood of $(0,\mu_{0})$ the set of solutions of $F(x,\mu)=0$ consists of two $C^{p-2}$ smooth curves $\Gamma_{1}$ and $\Gamma_{2}$ which intersect only at the point $(0,\mu_{0})$. Moreover, $\Gamma_{1}$ is the trivial branch
\[
\Gamma_{1}=\{(0,\mu)\},
\]
and $\Gamma_{2}$ can be parameterized as
\[
\Gamma_{2}:(x(\varepsilon),\mu(\varepsilon)),\ |\varepsilon|\ \text{small},
\]
with
\[
(x(0),\mu(0))=(0,\mu_{0}),\ x'(0)=x_{1}.
\]
\end{theorem}
}

The functional spaces in~\eqref{eq:Domain_space}, \eqref{eq:target_space} are chosen so that we consider solutions symmetric with respect to the direction of motion, i.e. the $x$-axis, and eliminate shifts to ensure the simple-eigenvalue property.
The regularity of spaces is chosen such that the operator $F$ in~\eqref{eq:F_nonlinear} is of the class $C^3$.

Indeed, let $\tilde F: X \times X_\sigma \times \R \to Y$, with \( X_\sigma:=H^2_{sym}(B),\)  denote the unreduced full operator in~\eqref{eq:F_nonlinear}, in which $\sigma$ is regarded as a free variable in the space $X_\sigma = H^2_{sym}(B)$, and not necessarily the solution to the elliptic problem. Then, the operator to which we apply the C.-R. bifurcation theorem is given by
\[
F= \tilde F \circ S.
\]
Therefore, to show that $F$ is a $C^3$ map, it is sufficient to show that both $\tilde F$ and $S$ are $C^3$ maps. In short, the key idea here is that for $R(\theta)$ close to the circle $R_0$, the operators with tildes behave like their regular counterparts, and the smoothness of $ S$ follows from the Implicit Function Theorem.

We now briefly justify the required regularity of the operators \(\tilde F\) and \(S\). We start with \(\tilde F\). Choose a neighborhood \(\mathcal U\subset H^{7/2}(S^1)\) of \(R_0\) such that, for every \(R\in \mathcal U\),
\[
        R(\theta)\ge c_0>0,
        \qquad
        r_\rho(\rho,\theta)
        =
        R_0+\chi'(\rho)(R(\theta)-R_0)
        \ge c_0>0 .
\]
Since \(H^{7/2}(S^1)\hookrightarrow C^{2,\alpha}(S^1)\), for every \(\alpha<1\), all coefficients appearing in
\[
        \tilde\nabla_R,\ 
        \tilde{\operatorname{div}}_R,\ 
        \tilde\Delta_R,\ 
        N[R],\ 
        \nu[R],\ 
        H[R],\ 
        w[R]
\]
are obtained from \(R,R_\theta,R_{\theta\theta}\), which are uniformly bounded, by smooth algebraic operations with denominators bounded away from zero. Moreover, because \(\chi=0\) near \(\rho=0\), the change of variables behaves at the origin like the regular polar coordinates. Hence, the tilde operators depend smoothly on \(R\) as functionals between the Sobolev spaces used in the definition of \(X\) and \(Y\). Together with the fact that the trace map \(H^2(B)\to H^{3/2}(\partial B)\) is continuous, and that \(D\) has the required \(C^4\)-regularity near \(m_0\), this implies that the unreduced operator $\tilde F$ is of class $C^3$ around the bifurcation point.

It remains only to justify the same regularity for the solution operator \(\sigma=S(x,K)\). Define
\[
G(x, \sigma, K)
=
\left(
        Z\tilde\Delta_R\sigma-\sigma+Pm,\,
        K N[R]\sigma-V e_1\cdot \nu[R]
\right).
\]
By the preceding coefficient argument, also $G \colon X \times X_\sigma \times \R \to X \times \R$ is of class $C^3$. At the homogeneous state,
\[
D_\sigma G(\sigma_0,m_0,0,R_0,K_0)[\tau]
=
\left(
        \frac{Z}{R_0^2}\Delta \tau-\tau,\,
        \frac{K_0}{R_0}\partial_\rho\tau\big|_{\partial B}
\right).
\]
This operator is an isomorphism from \(H^2_{\mathrm{sym}}(B)\) onto
\(L^2_{\mathrm{sym}}(B)\times H^{1/2}_{\mathrm{sym}}(\partial B)\), by standard elliptic regularity for the Neumann problem with a zeroth-order term. Therefore, the Implicit Function Theorem gives the required regularity of \(S\).

\section*{Author Contributions}
Leonid Berlyand, Oleksii Krupchytskyi, and Tim Laux each contributed to the formulation of the results, the development of their proofs, and the writing and editing of the manuscript.

\printbibliography
\pagestyle{plain}
\appendix
\newpage

\section{Expansions for general $D(m)$ in 2D}\label{app:general_expansions}

We expand the traveling wave solution as well as the domain of the cell in power series of $V$ for small $V$ around the rotationally symmetric resting cell configuration:
\begin{align}
	\sigma(r, \theta, V) &= \sigma_0 +\sigma_1(r,\theta) V+ \cdots\\
	m(r, \theta, V) &= m_0 + m_1(r, \theta) V + \cdots \\
	K(V) &= K_0 + K_1 V + K_2 V^2 + \cdots\\
	\rho(\theta,V) &= R_0 + \rho_1(\theta) V + \cdots,
\end{align}
where $\sigma_0, m_0, R_0$ are the steady resting state given by \eqref{eq:steady_state}. From the periodicity, each term can be expressed via Fourier series. From the symmetry around $x$-axis ($\theta=0$) we conclude that we can include only cosine modes in the expansions of the terms. Now, expand each term in Fourier modes
\begin{align}\label{eq:2d:sexp}
	\sigma_i(r, \theta) &= \sum\limits_{j=0}^\infty \sigma_{ij}(r)\cos(j \theta) \\
	\label{eq:2d:mexp}
	m_i(r, \theta) &= \sum\limits_{j=0}^\infty m_{ij}(r) \cos(j \theta) \\
	\label{eq:2d:rexp}
	\rho_i(\theta) &= \sum\limits_{j=0}^\infty \rho_{ij} \cos( j\theta),
\end{align}
for Fourier coefficients $\sigma_{ij}(r), m_{ij}(r)$, and $\rho_{nm}$. We will use the symmetry of the cell with respect to direction change, that is:
\begin{align}
	\label{eq:2d:ssym}
	\sigma(r, \theta, V) &= \sigma(r, \pi-\theta, -V)\\
	\label{eq:2d:msym}
	m(r, \theta, V) &= m(r, \pi-\theta, -V)\\
	\label{eq:2d:rsym}
	\rho(\theta, V) &= \rho(\pi-\theta, -V) \\
	\label{eq:2d:ksym}
	K(V) &= K(-V).
\end{align}

Substituting the expansion \eqref{eq:2d:sexp} into \eqref{eq:2d:ssym} and comparing the terms of like powers of $V$ and Fourier modes:
\begin{align}
\begin{split}\label{eq:2D:parity}
	\sigma_{ij}(r)\cos(j\theta) V^i &= \sigma_{ij}(r)\cos(j(\pi-\theta)) (-V)^i = (-1)^{i-j}
		\sigma_{ij}(r) \cos(j\theta) V^i.
\end{split}
\end{align}
From this we conclude that $\sigma_{ij}=0$ if $i$ and $j$ have different parity. The same argument yields $m_{ij}=0$ and $\rho_{ij}=0$ if $i$ and $j$ have different parity as well. From the symmetry, we can conclude that $K_i=0$ for all odd $i$.  

Next, we can eliminate the translation of the cell by imposing $\rho_{i1}=0$, that is there are no $\cos \theta$ modes in the expansions of the boundary, as such modes would correspond to the shifts in the $x$-axis.   

Therefore we have the following ansatz:
\begin{align}
	\label{eq:ans:s1}
	\sigma_1(r, \theta) &= \sigma_{11}(r)\cos(\theta)\\
	\label{eq:ans:s2}
	\sigma_2(r, \theta) &= \sigma_{20}(r) + \sigma_{22}\left(r\right)\cos(2\theta)\\
	\label{eq:ans:s3}
	\sigma_3(r, \theta) &= \sigma_{31}(r)\cos(\theta) + \sigma_{33}(r)\cos(3\theta)\\
	\label{eq:ans:m1}
	m_1(r, \theta) &= m_{11}(r)\cos(\theta)\\
	\label{eq:ans:m2}
	m_2(r, \theta) &= m_{20}(r) + m_{22}(r)\cos(2\theta)\\
	\label{eq:ans:m3}
	m_3(r, \theta) &= m_{31}(r)\cos(\theta) + m_{33}(r)\cos(3\theta)\\
	\label{eq:ans:rho1}
	\rho_1(\theta)&=0\\
	\label{eq:ans:rho2}
	\rho_2(\theta) &= \rho_{20}+\rho_{22}\cos(2\theta)\\
	\label{eq:ans:rho3}
	\rho_3(\theta) &= \rho_{33}\cos(3\theta).
	\end{align}

\subsection{Derivation of ODEs for expansion coefficients of Traveling Waves}

\subsubsection{Expansion of \eqref{eq:1_2d} in $V$:}

This is a linear equation which yields a simple expansion:
\begin{align}
	\label{eq:2d:1eq_0ord}
Z\Delta \sigma_0 = \sigma_0-Pm_0\\
\label{eq:2d:1eq_1ord}
Z\Delta \sigma_1 = \sigma_1-Pm_1\\
\label{eq:2d:1eq_2ord}
Z\Delta \sigma_2 = \sigma_2-Pm_2\\
\label{eq:2d:1eq_3ord}
Z\Delta \sigma_3 = \sigma_3-Pm_3
\end{align}

\subsubsection{Expansion of \eqref{eq:2_2d} in $V$:}
Now the RHS of \eqref{eq:2_2d} requires explicit computation. The linear expansions in first three orders reads
\begin{align}
	\label{eq:meq:0ord}
	0 =& \nabla\cdot(D(m_0) \nabla m_0 - K_0 \nabla(m_0 \nabla \sigma_0)) \\
	\label{eq:meq:1ord}
	0 =& D(m_0) \Delta m_1 - K_0 m_0 \Delta \sigma_1 \\
	\label{eq:meq:2ord}
	0 =& D(m_0) \Delta m_2 + \frac{1}{2} D'(m_0) \Delta m_1^2 + \mathbf e_1 \cdot\nabla m_1 - K_0 \nabla \cdot (m_0 \nabla \sigma_2+m_1 \nabla \sigma_1)\\
	\label{eq:meq:3ord}
	0 =& D(m_0) \Delta m_3 + D'(m_0) \Delta(m_1 m_2) + \frac{D''(m_0)}{6}\Delta(m_1^3) + \mathbf e_1 \cdot \nabla m_2 \\
	  &-K_2 m_0 \Delta\sigma_1 - K_0 \nabla \cdot (m_2 \nabla \sigma_1 + m_1 \nabla \sigma_2+ m_0 \nabla\sigma_3) \nonumber
\end{align}

\subsubsection{Expansion of \eqref{eq:3_2d} in $V$:}\label{subsubsec:exp_mbc}

First, we must expand the unit normal vector $\nu(\theta)$. The unit vector normal to the cell domain $\Omega(\theta)$ is given by:
\begin{equation}
	\nu(\theta) = \frac{\Omega(\theta)}{\sqrt{\Omega_\theta(\theta)^2+ \Omega(\theta)^2}} e_r - \frac{\Omega_\theta(\theta)}{\sqrt{\Omega_\theta(\theta)^2+ \Omega(\theta)^2}} e_\theta,
\end{equation}
where $e_r, e_\theta$ are the unit vectors in the radial and angular directions.
Expanding the denominator, one can show that:
\begin{equation}
	\nu(\theta) = e_r + \frac{1}{R_0}(\rho_{2,\theta}V^2+\rho_{3,\theta}V^3) e_\theta + O(V^4).
\end{equation}

Next, the expansion of the gradient around $r=R_0$ and $V=0$ in polar coordinates yields
\begin{align*}
	&\left.\nabla_{(r, \theta)} m(r, \theta) \right|_{r=\rho(\theta, V)} 
    \\&=
    \Big(m_{0,r}+m_{1r}V + (m_{2,r}+m_{0rr}\rho_2)V^2+(m_{3,r}+m_{0,rr}\rho_3 + m_{1,rr}\rho_2)V^3 \Big)e_r+\\
	\\&+\frac{1}{R_0} \Big(m_{0, \theta} + m_{1,\theta} V +  (m_{2,\theta}+ m_{0,r \theta})V^2 + (m_{3,\theta}+m_{0,r \theta} \rho_3 { - \frac{1}{R_0} }m_{1,\theta} \rho_{2}) V^3\Big) e_\theta + O(V^4).
\end{align*}
Combining the expressions above we obtain
\begin{align*}
	\nu(\theta) \cdot \nabla_{(r, \theta)} m(r, \theta, V) = m_{0r} + m_{1r}V + (m_{2r}+m_{0rr} \rho_2 + \frac{1}{R_0^2} m_{0\theta}2\rho_{2\theta})V^2 +\\
	+(m_{3,r}+m_{0,rr}\rho_3+m_{1,rr}\rho_2+\frac{1}{R_0} m_{0,\theta}\rho_{3,\theta}-\frac{1}{R_0^2} m_{1,\theta}\rho_{2,\theta})V^3 + O(V^4).
\end{align*}

Recall that steady state is homogeneous, thus all partial derivatives of zeroth order expansions are zero. This allows to derive the following boundary conditions:
\begin{align}
	\label{eq:mbc:0ord}
	m_{0,r}(R_0, \theta) & = 0, \\
	\label{eq:mbc:1ord}
	m_{1,r}(R_0, \theta) & = 0, \\
	\label{eq:mbc:2ord}
	m_{2,r}(R_0, \theta) & = 0, \\
	\label{eq:mbc:3ord}
	m_{3,r}(R_0, \theta) & = -m_{1,rr}(R_0, \theta)\rho_2(\theta) {+} \frac{1}{R_0^2} m_{1,\theta}(R_0,\theta)\rho_{2,\theta}(\theta).
\end{align}

\subsubsection{Expansion of \eqref{eq:4_2d} in $V$:}
The LHS of \eqref{eq:4_2d} expands as 
\begin{align*}
	\sigma(R(\theta), \theta) =& \sigma_0(R_0, \theta)+ \sigma_1(R_0, \theta) V + (\sigma_2(R_0, \theta)+\sigma_{0r}(R_0,\theta)\rho_2) V^2  \\
	&+(\sigma_3(R_0, \theta)+ \sigma_{0r}(R_0, \theta)\rho_3+\sigma_{1r}(R_0, \theta)\rho_2) V^3 + O(V^4).
\end{align*}
To expand the RHS  \eqref{eq:4_2d} we must first expand the volume of the region $\Omega$ in $V$:
\begin{equation}
	1- | \Omega(t)| = 1- \pi R_0^2 - {R_0}\int_0^{2 \pi} \rho_2(\theta) d \theta V^2 + O(V^4).
\end{equation}
We also expand the curvature $H$ in $V$
\begin{equation}
H = \frac{1}{ R_0} - \frac{1}{R_0^2} (\rho_2 + \rho_{2\theta \theta}) V^2 - \frac{1}{R_0^2}(\rho_3 + \rho_{3\theta\theta}) V^3 + O(V^4).
\end{equation}

Thus, we derive the following expansions
\begin{align}
	\label{eq:sigma:0ord}
	\sigma_0(R_0, \theta) & = 1- \pi R_0^2 -\frac{\gamma}{ R_0}, \\
	\label{eq:sigma:1ord}
	\sigma_1(R_0, \theta) & = 0, \\
	\label{eq:sigma:2ord}
	\sigma_2(R_0, \theta) & = -{ R_0}\int_0^{2 \pi} \rho_2(\theta) d \theta {+} \frac{\gamma}{R_0^2}(\rho_2+\rho_{2\theta \theta}), \\
	\label{eq:sigma:3ord}
	\sigma_3(R_0, \theta) & = -\sigma_{1r}(R_0,\theta) \rho_2 {+} \frac{\gamma}{R_0^2}(\rho_3+\rho_{3\theta\theta}).
\end{align}

\subsubsection{Expansion of \eqref{eq:5_2d} in $V$:}

The expansion of the kinematic boundary condition \eqref{eq:5_2d} for $\sigma$ is similar to the expansion of the BC \eqref{eq:3_2d} for $m$ in section \ref{subsubsec:exp_mbc} with the only difference that we have to take into account the correct powers of the expansion of $K$.
 We will derive the following boundary conditions:
 \begin{align}
	\label{eq:K_sigma:0ord}
	K_0 \sigma_{0,r}(R_0, \theta) & = 0, \\
	\label{eq:K_sigma:1ord}
	K_0 \sigma_{1,r}(R_0, \theta) & = 1, \\
	\label{eq:K_sigma:2ord}
	K_0 \sigma_{2,r}(R_0, \theta) & = 0,\\
	\label{eq:K_sigma:3ord}
	K_0 (\sigma_{3,r}(R_0, \theta) + \sigma_{1,rr}(R_0, \theta) \rho_2(\theta) {-} \frac{1}{R_0^2}\sigma_{1,\theta}(R_0, \theta) \rho_{2,\theta}(\theta)) + K_2 \sigma_{1r} & = 0.
\end{align}
{
\subsubsection{Expansion of \eqref{eq:myosin_mass} in $V$:}
 The integral constraint can be written in polar coordinates as 
 \begin{equation*}
   \int_0^{2\pi}\int_0^{\rho(\theta,V)} m(r,\theta,V)\,r\,dr\,d\theta=1. 
  \end{equation*} 
 Using the expansions of $m$ and $\rho$ in $V$, we obtain 
 \begin{align*} 
  1 =& \int_0^{2\pi}\int_0^{R_0} \left( m_0+m_1V+m_2V^2+m_3V^3 \right)r\,dr\,d\theta \\ &+ R_0\int_0^{2\pi} \left( m_0\rho_2V^2+ \left(m_0\rho_3+m_1(R_0,\theta)\rho_2\right)V^3 \right)d\theta +O(V^4). 
\end{align*}
 Thus, comparing terms of the same order in $V$ gives 
 \begin{align} \label{eq:mass:0ord} 
  \pi R_0^2 m_0 &= 1,\\ 
\label{eq:mass:1ord} \int_B m_1\,dx &=0,\\ 
\label{eq:mass:2ord} \int_B m_2\,dx + m_0R_0\int_0^{2\pi}\rho_2(\theta)\,d\theta &=0,\\ 
\label{eq:mass:3ord} \int_B m_3\,dx + R_0\int_0^{2\pi}m_1(R_0,\theta)\rho_2(\theta)\,d\theta + m_0R_0\int_0^{2\pi}\rho_3(\theta)\,d\theta &=0. 
\end{align} 
Using the ansatz \eqref{eq:ans:m1}--\eqref{eq:ans:rho3}, equations \eqref{eq:mass:1ord} and \eqref{eq:mass:3ord} are automatically satisfied, while \eqref{eq:mass:2ord} reduces to \begin{equation} \label{eq:mass:2ord_fourier} \int_0^{R_0} m_{20}(r)r\,dr + m_0R_0\rho_{20} = 0. \end{equation}

}

\section{Proof of the lemmas}\label{app:lemmas}

\begin{proof}[Proof of Lemma~\ref{lem:first_ord}]

Let $B$ be the disk of radius $R_0$ centered at the origin. 
Then, using formulas from Appendix \ref{app:general_expansions} the first order expansion reads: 
        \begin{equation}
  \left\{
  \begin{aligned}
    Z\Delta\sigma_{1}           &= \sigma_{1}-Pm_{1}
                                && \text{in } B,\\
    0                           &= D(m_{0})\Delta m_{1}
                                   -K_{0}m_{0}\,\Delta\sigma_{1}
                                && \text{in } B,\\
    m_{1r}(R_{0},\theta)        &= 0
                                &&\text{on } \partial B, \\
    \sigma_{1}(R_{0},\theta)   &= 0
                                && \text{on } \partial B,\\
    K_{0}\sigma_{1r}(R_{0},\theta) &= \cos \theta
                                &&\text{on } \partial B.
  \end{aligned}
  \right.
\end{equation}

    If we make the change of variables $m_1 = \frac{1}{D(m_0)} \hat m_1(x)$ as in \eqref{lem:1stord_2d} we get 
    \begin{equation}
  \left\{
  \begin{aligned}
    Z\Delta\hat\sigma_{1} &= \hat\sigma_{1} - P\hat m_{1} && \text{in } B,\\
    0 &= \Delta\hat m_{1} - \hat K_{0} m_{0}\,\Delta\hat\sigma_{1} && \text{in } B,\\
    \hat m_{1r}(R_{0},\theta) &= 0 && \text{on } \partial B,\\
    \hat\sigma_{1}(R_{0},\theta)&= 0&& \text{on } \partial B,\\
    \hat K_{0}\,\hat\sigma_{1r}(R_{0},\theta)&= \cos \theta&& \text{on } \partial B.
  \end{aligned}
  \right.
  \label{eq:1ordhat}
\end{equation}
 
Now system \eqref{eq:1ordhat} is independent of $D(m)$. Therefore the first order perturbation can be expressed in the form \eqref{lem:1stord_2d}. Moreover, the explicit solution of the system \eqref{eq:1ordhat} in terms of Bessel function can be found using Fourier analysis. Indeed, using the ansatz 
\begin{equation}
    \hat m_{1} = \hat m_{11}(\rho) \cos \theta, \hat \sigma_1 = \hat \sigma_{11}\left(\rho\right)\cos\theta.
\end{equation}
After substituting this into the system \eqref{eq:1ordhat} we derive a system of ODEs that has an explicit solution
\begin{align}
    \hat \sigma_{11}(\rho)&=
\frac{1}{P \hat K_{0} m_{0}-1}
\Big(P m_0\rho
-\frac{R_0}{\hat K_{0}}
\frac{J_{1}\bigl(\alpha/R_0\rho\bigr)}{\alpha J_{1}'(\alpha )}
\Big),\\
    \hat m_{11}(\rho) &= \hat K_0 m_0 \hat \sigma_{11} - m_0 \rho
\end{align}

Note, that the same formulas will be recovered in the proof of the simple zero eigenvalue in the Crandall-Rabinowitz theorem up to the change of coordinates between the ball of radius $R_0$ and a unit ball.
\end{proof}

\begin{proof}[Proof of Lemma~\ref{lem:2ndorder}]
Combine the expansions \eqref{eq:2d:1eq_2ord}, \eqref{eq:meq:2ord}, \eqref{eq:mbc:2ord},\eqref{eq:sigma:2ord}, \eqref{eq:K_sigma:2ord} we get the following system {for the unknown $m_2, \sigma_2, \rho_2$}

\begin{subequations}\label{eq:2_system_clean}
  \begin{numcases}{}
    Z\,\Delta\sigma_{2} \;=\; \sigma_{2} - P\,m_{2}
      \hfill \text{in } B \label{eq:2_a}\\[2pt]
    0 \;=\; D(m_{0})\,\Delta m_{2}
         + \frac12 D'(m_{0})\,\Delta (m_{1}^2)
         + \mathbf e_{1}\!\cdot\!\nabla m_{1} \notag\\
    \hspace{2.2em}
    {}- K_{0}\!\bigl(
              m_{0}\,\Delta\sigma_{2}
              + \nabla m_{1}\!\cdot\!\nabla\sigma_{1}
              + m_{1}\,\Delta\sigma_{1}
           \bigr)
      \hfill \text{in } B \label{eq:2_b}\\[2pt]
    m_{2,r}(R_{0},\theta) = 0
      \hfill \text{on } \partial B \label{eq:2_c}\\[2pt]
    \sigma_{2}(R_{0},\theta) =
      -{ R_0}\int_0^{2 \pi} \rho_2(\theta) d \theta {+} \frac{\gamma}{R_0^2}(\rho_2+\rho_{2\theta \theta})
      \hfill \text{on } \partial B \label{eq:2_d}\\[2pt]
    K_{0}\,\sigma_{2,r}(R_{0},\theta) \;=\; 0
      \hfill \text{on } \partial B \label{eq:2_e}
  \end{numcases}
\end{subequations}

Now we can use Lemma \ref{lem:first_ord} to substitute the formula for $m_1, \sigma_1, K_0$ in the expansions~\eqref{eq:2_a}-\eqref{eq:2_e} and derive the following system

\begin{subequations}\label{eq:2_system}
  \begin{numcases}{}
    Z\,\Delta\sigma_{2} = \sigma_{2} - P\,m_{2} \hfill \text{in } B \label{eq:2_a}\\[2pt]
    0 =
      \Delta m_{2}
      + \frac{1}{2}\,\frac{D'(m_{0})}{D(m_{0})^{3}}\Delta(\hat m_{1}^{2})
      + \frac{\mathbf e_{1}\!\cdot\!\nabla\hat m_{1}}{D(m_{0})^{2}}
      - \hat K_{0} m_{0}\,\Delta\sigma_{2} \notag\\
      \hspace{2.2em}
      {}- \frac{\hat K_{0}}{D(m_{0})^{2}}
        \bigl(\nabla\hat m_{1}\!\cdot\!\nabla\hat\sigma_{1}
              + \hat m_{1}\,\Delta\hat\sigma_{1}\bigr)
      \hfill \text{in } B \label{eq:2_b}\\[2pt]
    m_{2,r}(R_{0},\theta) = 0 \hfill \text{on } \partial B \label{eq:2_c}\\[2pt]
    \sigma_{2}(R_{0},\theta) =
      -{ R_0}\int_0^{2 \pi} \rho_2(\theta) d \theta {+} \frac{\gamma}{R_0^2}(\rho_2+\rho_{2\theta \theta})
      \hfill \text{on } \partial B \label{eq:2_d}\\[2pt]
    \sigma_{2,r}(R_{0},\theta) = 0 \hfill \text{on } \partial B \label{eq:2_e}
  \end{numcases}
\end{subequations}

Now introduce the two systems {for $m_{2A}, \sigma_{2A}, \rho_{2A}$ and  $m_{2B}, \sigma_{2B}, \rho_{2B}$ respectively} 
\begin{subequations}\label{eq:2a_system}
  \begin{numcases}{}
    Z\,\Delta \sigma_{2A} \;=\; \sigma_{2A} - P\,m_{2A}
      \hfill \text{in } B \label{eq:2a_a}\\[2pt]
    0 \;=\; \Delta m_{2A}
      - \hat K_0 m_0\,\Delta \sigma_{2A}
      + \mathbf e_1 \!\cdot\! \nabla \hat m_1
      - \hat K_0\, \nabla \hat m_1 \!\cdot\! \nabla \hat \sigma_1
      - \hat K_0\, \hat m_1 \,\Delta \hat \sigma_1
      \hfill \text{in } B \label{eq:2a_b}\\[2pt]
    m_{2A,r}(R_0,\theta) \;=\; 0
      \hfill \text{on } \partial B \label{eq:2a_c}\\[2pt]
    \sigma_{2A}(R_0,\theta) =
      -{ R_0}\int_0^{2 \pi} \rho_{2A}(\theta) d \theta {+} \frac{\gamma}{R_0^2}(\rho_{2A}+\rho_{2A,\theta \theta})
      \hfill \text{on } \partial B \label{eq:2a_d}\\[2pt]
    \sigma_{2A,r}(R_0,\theta) \;=\; 0
      \hfill \text{on } \partial B \label{eq:2a_e}
  \end{numcases}
\end{subequations}

\begin{subequations}\label{eq:2b_system}
  \begin{numcases}{}
    Z\,\Delta \sigma_{2B} \;=\; \sigma_{2B} - P\,m_{2B}
      \hfill \text{in } B \label{eq:2b_a}\\[2pt]
    0 \;=\; \Delta m_{2B}
      - \hat K_0 m_0\,\Delta \sigma_{2B}
      + \frac{1}{2}\,\Delta\!\bigl(\hat m_{1}^{2}\bigr)
      \hfill \text{in } B \label{eq:2b_b}\\[2pt]
    m_{2B,r}(R_0,\theta) \;=\; 0
      \hfill \text{on } \partial B \label{eq:2b_c}\\[2pt]
    \sigma_{2B}(R_0,\theta) \;=\;
      -{ R_0}\int_0^{2 \pi} \rho_{2B}(\theta) d \theta {+} \frac{\gamma}{R_0^2}(\rho_{2B}+\rho_{2B,\theta \theta})
      \hfill \text{on } \partial B \label{eq:2b_d}\\[2pt]
    \sigma_{2B,r}(R_0,\theta) \;=\; 0
      \hfill \text{on } \partial B \label{eq:2b_e}
  \end{numcases}
\end{subequations}

By using the ansatz from Appendix \ref{app:general_expansions}, we can search for the solutions $m_{2A}, m_{2B}$ in the following form
\begin{align}
    m_{2A}(r, \theta) &= m_{20A}(r) + m_{22A}(r)\cos(2\theta)\\
    \sigma_{2A}(r, \theta) &= \sigma_{20A}(r) + \sigma_{22A}(r)\cos(2\theta)\\
    \rho_{2A}(\theta) &= \rho_{20A} + \rho_{22A} \cos (2 \theta),
\end{align}
and
\begin{align}
    m_{2B}(r, \theta) &= m_{20B}(r) + m_{22B}(r)\cos(2\theta)\\
    \sigma_{2B}(r, \theta) &= \sigma_{20B}(r) + \sigma_{22B}(r)\cos(2\theta)\\
    \rho_{2B}(\theta) &= \rho_{20B} + \rho_{22B} \cos (2 \theta),
\end{align}
where the equations for the Fourier coefficients are derived by plugging the expansion above into \eqref{eq:2a_system} and \eqref{eq:2b_system} respectively. 
{
 The four systems for $(m_{20A}, \sigma_{20A}, \rho_{20A})$, $(m_{20B}, \sigma_{20B}, \rho_{20B})$, $(m_{22A}, \sigma_{22A}, \rho_{22A})$, $(m_{22B}, \sigma_{22B}, \rho_{22B})$ are presented below. All systems are solved as follows. First, one solves a well posed system (a),(b),(c),(e) for $m$ and $\sigma$ and then recovers the value of the corresponding $\rho$ via the boundary condition (d).

\begin{subequations}\label{eq:mode0A}
  \begin{numcases}{}
    Z\!\left(\sigma_{20A}''+\frac{1}{r}\sigma_{20A}'\right) \;=\; \sigma_{20A} - P\,m_{20A},
      \qquad 0<r<R_0, \label{eq:mode0A_a}\\[6pt]
    0 \;=\; \left(m_{20A}''+\frac{1}{r}m_{20A}'\right)
             - \hat K_0 m_0\!\left(\sigma_{20A}''+\frac{1}{r}\sigma_{20A}'\right)
             + F_A^{(0)}(r),
      \qquad 0<r<R_0, \label{eq:mode0A_b}\\[6pt]
    m_{20A}'(R_0)=0, \label{eq:mode0A_c}\\[4pt]
    \sigma_{20A}(R_0)= -2\pi {R_0}\rho_{20A}{+}\dfrac{\gamma}{R_0^{2}}\,\rho_{20A}, \label{eq:mode0A_d}\\
    \sigma_{20A}'(R_0)=0, \label{eq:mode0A_e}
  \end{numcases}
\end{subequations}
where 
\begin{equation*}
  F_A^{(0)}(r)
= \frac12\!\left(\hat m_{11}'+\frac{\hat m_{11}}{r}\right)
-\frac{\hat K_0}{2}\!\left(\hat m_{11}'\,\hat \sigma_{11}' + \hat m_{11}\hat \sigma_{11}^{''}+\frac{\hat m_{11}\hat \sigma_{11}^{'}}{r}\right),
\end{equation*}

\begin{subequations}\label{eq:mode0B}
  \begin{numcases}{}
    Z\!\left(\sigma_{20B}''+\frac{1}{r}\sigma_{20B}'\right) \;=\; \sigma_{20B} - P\,m_{20B},
      \qquad 0<r<R_0, \label{eq:mode0B_a}\\[6pt]
    0 \;=\; \left(m_{20B}''+\frac{1}{r}m_{20B}'\right)
             - \hat K_0 m_0\!\left(\sigma_{20B}''+\frac{1}{r}\sigma_{20B}'\right)
             + F_B^{(0)}(r),
      \qquad 0<r<R_0, \label{eq:mode0B_b}\\[6pt]
    m_{20B}'(R_0)=0, \label{eq:mode0B_c}\\[4pt]
    \sigma_{20B}(R_0)= -2\pi {R_0} \rho_{20B}{+}\dfrac{\gamma}{R_0^{2}}\,\rho_{20B},\label{eq:mode0B_d} \\
    \sigma_{20B}'(R_0)=0, 
  \end{numcases}
\end{subequations}
where 
\begin{equation*}
  F_B^{(0)}(r) = \frac12\left(\hat m_{11}'^2+ \hat m_{11} \hat m_{11}'' + \frac1r \hat m_{11}\hat m_{11}'\right),
\end{equation*}

\begin{subequations}\label{eq:mode2A}
  \begin{numcases}{}
    Z\!\left(\sigma_{22A}''+\frac{1}{r}\sigma_{22A}'-\frac{4}{r^{2}}\sigma_{22A}\right)
      \;=\; \sigma_{22A} - P\,m_{22A},
      \qquad 0<r<R_0, \label{eq:mode2A_a}\\[8pt]
    0 \;=\; \left(m_{22A}''+\frac{1}{r}m_{22A}'-\frac{4}{r^{2}}m_{22A}\right)
             - \hat K_0 m_0\!\left(\sigma_{22A}''+\frac{1}{r}\sigma_{22A}'-\frac{4}{r^{2}}\sigma_{22A}\right)
             + F_A^{(2)}(r),
      \qquad 0<r<R_0, \label{eq:mode2A_b}\\[8pt]
    m_{22A}'(R_0)=0, \label{eq:mode2A_c}\\[4pt]
    \sigma_{22A}(R_0)= {-}\dfrac{3\gamma}{R_0^{2}}\,\rho_{22A},\label{eq:mode2A_d} \\
    \sigma_{22A}'(R_0)=0, \label{eq:mode2A_e}
  \end{numcases}
\end{subequations}
where
\begin{equation*}
  F_A^{(2)}(r)
= \frac12\left(\hat m_{11}'-\frac{\hat m_{11}}{r}\right)
-\frac{\hat K_0}{2}\left(\hat m_{11}'\hat \sigma_{11}' + \hat m_{11}\hat \sigma_{11}'' + \frac1r\hat m_{11}\hat \sigma_{11}'-\frac{2}{r^2}\hat m_{11}\hat \sigma_{11}\right),
\end{equation*}

\begin{subequations}\label{eq:mode2B}
  \begin{numcases}{}
    Z\!\left(\sigma_{22B}''+\frac{1}{r}\sigma_{22B}'-\frac{4}{r^{2}}\sigma_{22B}\right)
      \;=\; \sigma_{22B} - P\,m_{22B},
      \qquad 0<r<R_0, \label{eq:mode2B_a}\\[8pt]
    0 \;=\; \left(m_{22B}''+\frac{1}{r}m_{22B}'-\frac{4}{r^{2}}m_{22B}\right)
             - \hat K_0 m_0\!\left(\sigma_{22B}''+\frac{1}{r}\sigma_{22B}'-\frac{4}{r^{2}}\sigma_{22B}\right)
             + F_B^{(2)}(r),
      \qquad 0<r<R_0, \label{eq:mode2B_b}\\[8pt]
    m_{22B}'(R_0)=0, \label{eq:mode2B_c}\\[4pt]
    \sigma_{22B}(R_0)= {-}\dfrac{3\gamma}{R_0^{2}}\,\rho_{22B}, \label{eq:mode2B_d} \\
    \sigma_{22B}'(R_0)=0, \label{eq:mode2B_e}
  \end{numcases}
\end{subequations}
where
\begin{equation*}
  F_B^{(2)}(r) = \frac12\left(\hat m_{11}'^2+ \hat m_{11} \hat m_{11}'' + \frac1r \hat m_{11}\hat m_{11}'\right) - \frac{\hat m_{11}^2}{r^2}.
\end{equation*}
}
\end{proof}
\begin{proof}[Proof of Lemma~\ref{lem:3rdorder}]
   
    Using the expansions from Appendix~\ref{app:general_expansions}, we consider the third order expansion of the system~\eqref{eq:1_2d}--\eqref{eq:5_2d} which is given by 
\begin{subequations}\label{eq:third_order_system}
  \begin{numcases}{}
    Z\,\Delta \sigma_{3} \;=\; \sigma_{3} - P\,m_{3}
      \hfill \text{in } B \label{eq:3_a}\\[2pt]
    0 \;=\; D(m_{0})\,\Delta m_{3}
      + \,D'(m_{0})\,\Delta(m_{1}m_{2})
      + \frac{1}{6}\,D''(m_{0})\,\Delta(m_{1}^{3})
      + \mathbf e_{1}\!\cdot\!\nabla m_{2} \notag\\
    \hspace{2.2em}
    {}- K_{2}\,m_{0}\,\Delta\sigma_{1}
      - K_{0}\,\nabla\!\cdot\!\bigl(m_{2}\nabla\sigma_{1}
                                    + m_{1}\nabla\sigma_{2}
                                    + m_{0}\nabla\sigma_{3}\bigr)
      \hfill \text{in } B \label{eq:3_b}\\[2pt]
    m_{3,r} = -\,m_{1,rr}\,\rho_{2}
      {+} \frac{1}{R_{0}^{2}}\,m_{1,\theta}\,\rho_{2,\theta}
      \hfill \text{on } \partial B \label{eq:3_c}\\[2pt]
    \sigma_{3} \;=\; -\,\sigma_{1,r}\,\rho_{2}
      {+} \frac{\gamma}{R_{0}^{2}}\bigl(\rho_{3}+\rho_{3,\theta\theta}\bigr)
      \hfill \text{on } \partial B \label{eq:3_d}\\[2pt]
    K_{0}\!\left(
      \sigma_{3,r}
      + \sigma_{1,rr}\,\rho_{2}
      {-} \frac{1}{R_{0}^{2}}\sigma_{1,\theta}\,\rho_{2,\theta}
    \right)
    + K_{2}\,\sigma_{1,r} \;=\; 0
      \hfill \text{on } \partial B \label{eq:3_e}
  \end{numcases}
\end{subequations}

    After that we use the ansatz for $m_3, \sigma_3$, and $\rho_3$ given by \eqref{eq:ans:m3}, \eqref{eq:ans:s3}, \eqref{eq:ans:rho3} and combining the terms in front and by collecting $\cos \theta$ terms we derive the following system for $m_{31}, \sigma_{31}$,
\begin{subequations}\label{eq:ms31_system}
  \begin{numcases}{}
    Z\left(\sigma_{31}''+\frac{1}{r}\sigma_{31}'-\frac{1}{r^{2}}\sigma_{31}\right)-\sigma_{31}
      \;=\; -P\,m_{31}
      \hfill \text{for } 0<r<R_0 \label{eq:ms31_a}\\[2pt]
    D(m_{0})\left(m_{31}''+\frac{1}{r}m_{31}'-\frac{1}{r^{2}}m_{31}\right)
    - K_{0}m_{0}\left(\sigma_{31}''+\frac{1}{r}\sigma_{31}'-\frac{1}{r^{2}}\sigma_{31}\right)
    = \notag\\
    \hspace{2.2em}
    K_{2}m_{0}\left(\sigma_{11}''+\frac{1}{r}\sigma_{11}'-\frac{1}{r^2}\sigma_{11}\right) - f(r)
      \hfill \text{for } 0<r<R_0 \label{eq:ms31_b}\\[2pt]
    m'_{31}(R_0) + \bigl(\rho_{20}+ \frac{1}{2}\rho_{22}\bigr) m_{11}''(R_0)
      {-} \frac{\rho_{22}}{R_0^2}\, m_{11}(R_0) \;=\; 0
  \label{eq:ms31_c}\\[2pt]
    \sigma_{31}(R_0) + \bigl(\rho_{20} + \frac12 \rho_{22}\bigr) \sigma'_{11}(R_0) \;=\; 0
   \label{eq:ms31_d}\\[2pt]
    K_2\,\sigma'_{11}(R_0)
      + K_0\!\left(\sigma'_{31}(R_0)
                   + \bigl(\rho_{20}+ \frac{1}{2}\rho_{22}\bigr) \sigma''_{11}(R_0)
                   {-}\frac{\rho_{22}}{R_0^2}\, \sigma_{11}(R_0)\right)=0 \label{eq:ms31_e}\\[2pt]
    m_{31}(0) \;=\; 0\label{eq:ms31_f}\\[2pt]
    \sigma_{31}(0) \;=\; 0\label{eq:ms31_g}
  \end{numcases}
\end{subequations}

where $f(r)$ depends only on lower order terms and is given by 

\begin{align}\label{eq:f(s)}
&f\left(r\right) = \frac{D''(m_{0})}{8 r^{2}}\Bigl(
 - m_{11}^{3}
 + 6 r^{2} m_{11}(m_{11}')^{2}
 + 3 r\, m_{11}^{2}(m_{11}' + r\,m_{11}'')
\Bigr) \notag\\
&
+ \frac{D'(m_{0})}{8 r^{2}}\Bigl(
 - 8 m_{11} m_{20}
 - 4 m_{11} m_{22}
 + 8 r m_{20} m_{11}'
 + 4 r m_{22} m_{11}'
 + 8 r m_{11} m_{20}'\notag\\
&
 + 16 r^{2} m_{11}' m_{20}'
 + 4 r m_{11} m_{22}'
 + 8 r^{2} m_{11}' m_{22}' 
 + 8 r^{2} m_{20} m_{11}''\notag\\
 &
 + 4 r^{2} m_{22} m_{11}''
 + 8 r^{2} m_{11} m_{20}''
 + 4 r^{2} m_{11} m_{22}''
\Bigr) \notag\\
&
+ \frac{1}{8 r^{2}}\Bigl(
   8 r^{2} m_{20}'
 + 4 r^{2} m_{22}' 
 + 8 r m_{22} \notag \\
 &
 + K_0 \left(
   8 m_{20} \sigma_{11}
 - 4 m_{22} \sigma_{11}
 + 8 m_{11} \sigma_{22}
 - 8 r m_{20} \sigma_{11}'
 - 4 r m_{22} \sigma_{11}' \right.\notag\\
&
 - 8 r^{2} m_{20}' \sigma_{11}'
 - 4 r^{2} m_{22}' \sigma_{11}' 
 - 8 r m_{11} \sigma_{20}'
 - 4 r m_{11} \sigma_{22}'
 - 4 r^{2} m_{11}' \sigma_{22}'\notag\\
  &
\left.
 - 8 r^{2} m_{20} \sigma_{11}''
 - 4 r^{2} m_{22} \sigma_{11}''
 - 8 r^{2} m_{11}' \sigma_{20}'
 - 8 r^{2} m_{11} \sigma_{20}''
 - 4 r^{2} m_{11} \sigma_{22}''
 \right)
\Bigr).
\end{align}
Note that the direct substitution and collecting the alike terms together shows that $f(r)$ can be written as

\begin{equation}
    f(r) = \frac{D''(m_0)}{D(m_0)^3}f_1(r)+\frac{D'(m_0)^2}{D(m_0)^4}f_2(r)+\frac{D'(m_0)}{D(m_0)^3}f_3(r)+\frac{1}{D(m_0)^2}f_4(r),
\end{equation}
where $f_1, f_2, f_3, f_4$ are given by

\begin{align}\label{eq:f1}
f_1\left(r\right)=\frac{1}{8r^{2}}\hat m_{11}\,\Big(-\hat m_{11}^{2}
+ 6r^{2}(\hat m_{11}')^{2}
+ 3r\hat m_{11}(\hat m_{11}' + r\hat m_{11}'')\Big),
\end{align}
\begin{align}\label{eq:f2}
f_2\left(r\right)&=
\frac{1}{8 r^{2}}\Bigl(
 - 8 \hat m_{11} m_{20B}
 - 4 \hat m_{11} m_{22B}
 + 8 r m_{20B} \hat m_{11}'
 + 4 r m_{22B} \hat m_{11}'
 + 8 r \hat m_{11} m_{20B}'\notag\\
&
 + 16 r^{2} \hat m_{11}' m_{20B}'
 + 4 r \hat m_{11} m_{22B}'
 + 8 r^{2} \hat m_{11}' m_{22B}' 
 + 8 r^{2} m_{20B} \hat m_{11}''\notag\\
 &
 + 4 r^{2} m_{22B} \hat m_{11}''
 + 8 r^{2} \hat m_{11} m_{20B}''
 + 4 r^{2} \hat m_{11} m_{22B}''
\Bigr),
\end{align}
\begin{align}\label{eq:f3}
f_3\left(r\right)&=
\frac{1}{8 r^{2}}\Bigl(
 - 8 \hat m_{11} m_{20A}
 - 4 \hat m_{11} m_{22A}
 + 8 r m_{20A} \hat m_{11}'
 + 4 r m_{22A} \hat m_{11}'
 + 8 r \hat m_{11} m_{20A}'\notag\\
&
 + 16 r^{2} \hat m_{11}' m_{20A}'
 + 4 r \hat m_{11} m_{22A}'
 + 8 r^{2} \hat m_{11}' m_{22A}' 
 + 8 r^{2} m_{20A} \hat m_{11}''\notag\\
 &
 + 4 r^{2} m_{22A} \hat m_{11}''
 + 8 r^{2} \hat m_{11} m_{20A}''
 + 4 r^{2} \hat m_{11} m_{22A}''\notag \\
&
+ 
   8 r^{2} m_{20B}'
 + 4 r^{2} m_{22B}' 
 + 8 r m_{22B} \notag \\
 &
 + \hat K_0 \left(
   8 m_{20B} \hat \sigma_{11}
 - 4 m_{22B} \hat \sigma_{11}
 + 8 \hat m_{11} \sigma_{22B}
 - 8 r m_{20B} \hat \sigma_{11}'
 - 4 r m_{22B} \hat \sigma_{11}' \right.\notag\\
&
 - 8 r^{2} m_{20B}' \hat \sigma_{11}'
 - 4 r^{2} m_{22B}' \hat \sigma_{11}' 
 - 8 r \hat m_{11} \sigma_{20B}'
 - 4 r \hat m_{11} \sigma_{22B}'
 - 4 r^{2} \hat m_{11}' \sigma_{22B}'\notag\\
  &
\left.
 - 8 r^{2} m_{20B} \hat \sigma_{11}''
 - 4 r^{2} m_{22B} \hat \sigma_{11}''
 - 8 r^{2} \hat m_{11}' \sigma_{20B}'
 - 8 r^{2} \hat m_{11} \sigma_{20B}''
 - 4 r^{2} \hat m_{11} \sigma_{22B}''
 \right)
\Bigr).
\end{align}
\begin{align}\label{eq:f4}
&f_4\left(r\right) = \frac{1}{8r^{2}}\Bigl(
   8 r^{2} m_{20A}'
 + 4 r^{2} m_{22A}' 
 + 8 r m_{22A} \notag \\
 &
 + \hat K_0 \left(
   8 m_{20A} \hat \sigma_{11}
 - 4 m_{22A} \hat \sigma_{11}
 + 8 \hat m_{11} \sigma_{22A}
 - 8 r m_{20A} \hat \sigma_{11}'
 - 4 r m_{22A} \hat \sigma_{11}' \right.\notag\\
&
 - 8 r^{2} m_{20A}' \hat \sigma_{11}'
 - 4 r^{2} m_{22A}' \hat \sigma_{11}' 
 - 8 r \hat m_{11} \sigma_{20A}'
 - 4 r \hat m_{11} \sigma_{22A}'
 - 4 r^{2} \hat m_{11}' \sigma_{22A}'\notag\\
  &
\left.
 - 8 r^{2} m_{20A} \hat \sigma_{11}''
 - 4 r^{2} m_{22A} \hat \sigma_{11}''
 - 8 r^{2} \hat m_{11}' \sigma_{20A}'
 - 8 r^{2} \hat m_{11} \sigma_{20A}''
 - 4 r^{2} \hat m_{11} \sigma_{22A}''
 \right)
\Bigr).
\end{align}
The main idea here is to find $m_{31}$ in terms of $\sigma_{31}$ that will allow to reduce the system~\eqref{eq:ms31_system} to a single ODE for $\sigma_{31}$.  This ODE has an extra boundary condition and the solvability condition is found with the help of the test function trick. In general the solution for $m_{31}$ is obtained via the Green function, but we keep the two most important terms including $K_2$ and $\sigma_{31}$ explicit:

\begin{align}
    m_{31}(r) = \frac{K_0m_0}{D(m_0)} \sigma_{31}(r) + K_2\frac{m_0}{D(m_0)} \sigma_{11}(r) + C_1 r+ C_2 \frac1r \nonumber\\
    - \frac{1}{D(m_0)} \left( \frac{r}{2} \int_0^{r} f(s)\,ds- \frac{1}{2r} \int_0^r s^2 f(s)\,ds \right).
\end{align}
Now because of the regularity at $0$ we have $C_2=0$ and $C_1$ can be found from direct differentiation

\begin{align}
    C_1 = &m_{31}'(R_0) - \frac{K_0 m_0}{D(m_0)} \sigma_{31}'(R_0) - \frac{K_2 m_0}{D(m_0)} \sigma_{11}'(R_0) \nonumber\\ 
    &+\frac{1}{D(m_0)}\left(\frac{1}{2}\int_{0}^{R_0}f(s)\,ds+\frac{1}{2R_0^2}\int_{0}^{R_0} s^2f\left(s\right)ds  \right).
\end{align}
Finally, the Neumann boundary conditions for $\sigma_{31}$ and $m_{31}$ can be substituted and $C_1$ can be expressed as 

\begin{align}\label{eq:C1}
    &C_1 = -\bigl(\rho_{20}+\frac{1}{2}\rho_{22}\bigr) m_{11}''(R_0)
      {+} \frac{\rho_{22}}{R_0^2}m_{11}(R_0) - \frac{K_2 m_0}{D(m_0)} \sigma_{11}'(R_0)\nonumber \\
      & +\frac{K_0 m_0}{D(m_0)} \left( \frac{K_2}{K_0}\sigma'_{11}(R_0)
                   + \bigl(\rho_{20}+ \frac{1}{2}\rho_{22}\bigr) \sigma''_{11}(R_0)
                   {-} \frac{\rho_{22}}{R_0^2}\, \sigma_{11}(R_0)\right)\nonumber\\
      &+\frac{1}{D(m_0)}\left(\frac{1}{2}\int_{0}^{R_0}f(s)\,ds+\frac{1}{2R_0^2}\int_{0}^{R_0} s^2f\left(s\right)ds \right)\nonumber\\
   &= \bigl(\rho_{20}+\frac{1}{2}\rho_{22}\bigr)\left(\frac{K_0 m_0}{D(m_0)}  \sigma''_{11}(R_0)-m_{11}''(R_0)\right) 
      {+} \frac{\rho_{22}}{R_0^2}m_{11}(R_0)+
      \nonumber \\
   & +\frac{1}{D(m_0)}\left(\frac{1}{2}\int_{0}^{R_0}f(s)\,ds+\frac{1}{2R_0^2}\int_{0}^{R_0} s^2f\left(s\right)ds\right).
\end{align}

Now, we can substitute $m_{31}$ in the first equation of \eqref{eq:ms31_system} in terms of $\sigma_{31}$ and derive the following ODE for $\sigma_{31}$ with one extra boundary condition that will allow to determine $K_2$
\begin{subequations}\label{eq:s31_system}
  \begin{numcases}{}
    Z\!\left(\sigma_{31}''+\frac{1}{r}\sigma_{31}'-\frac{1}{r^{2}}\sigma_{31}\right)
    + \left(\frac{K_0 P m_0}{D(m_0)}-1\right)\sigma_{31}
    = -K_2\frac{P m_0}{D(m_0)}\sigma_{11}
     \notag\\
    - \frac{P}{D(m_0)} \left(\frac{1}{2r} \int_0^r s^2 f(s)\,ds +\frac{r}{2}\int_{r}^{R_0}f(s)\,ds+\frac{r}{2R_0^2}\int_{0}^{R_0} s^2f\left(s\right)ds\right)\notag \\
  -\,P\!\left[
      \bigl(\rho_{20}+ \frac{1}{2}\rho_{22}\bigr) \left(\frac{K_0 m_0}{D(m_0)} \sigma''_{11}(R_0)-m_{11}''(R_0)\right){+}\frac{\rho_{22}}{R_0^2}m_{11}(R_0)
      \right]r \text{ for } 0<r<R_0 \label{eq:s31_pde}\\[2pt]
    \sigma_{31}(R_0) + \bigl(\rho_{20}+\frac{1}{2}\rho_{22}\bigr)\,\sigma_{11}'(R_0) = 0\label{eq:s31_bc1}\\[2pt]
    K_2\,\sigma_{11}'(R_0)
    + K_0\!\left(
        \sigma_{31}'(R_0)
        + \bigl(\rho_{20}+\frac{1}{2}\rho_{22}\bigr)\sigma_{11}''(R_0)
        {-} \frac{\rho_{22}}{R_0^{2}}\,\sigma_{11}(R_0)
      \right) = 0  \label{eq:s31_bc2}\\[2pt]
    \sigma_{31}(0) = 0.\label{eq:s31_bc3}
  \end{numcases}
\end{subequations}

Now introduce the test function $U$ that solves the homogeneous equation 
\begin{subequations}
    \begin{numcases}{}
         Z(U''+ \frac 1r U'- \frac{1}{r^2} U) + \left(\frac{K_0Pm_0}{D(m_0)}-1\right) U = 0, \quad 0<r<R_0 \\
         U(0) = 0 \\
         U(R_0) = 1.
    \end{numcases}
\end{subequations}
The solution to which is explicitly given by the Bessel function
\begin{equation}
    U(r) = \frac{J_1(\alpha/R_0 r)}{J_1(\alpha)},
\end{equation}
where $\alpha = \dfrac{R_0}{\sqrt Z}\sqrt{\dfrac{K_0Pm_0}{D(m_0)}-1} =\dfrac{R_0}{\sqrt Z}\sqrt{{\hat K_0Pm_0}-1} $ is independent of $D(m)$.

Now if we integrate both sides of ODE in~\eqref{eq:s31_pde} and take into account the boundary condition we derive a linear equation for $K_2$. 
Indeed, the LHS of~\eqref{eq:s31_pde} reduces to
\begin{align}\label{eq:by_parts}
    &\int_0^{R_0} \left( Z\left(\sigma_{31}''+ \frac 1r \sigma_{31}'-\frac{1}{r^2} \sigma_{31}\right) + \left(\frac{K_0Pm_0}{D(m_0)}-1\right) \sigma_{31} \right) U r\,dr \nonumber
    \\\nonumber
     &=Z\left[- r \sigma_{31}U'+r U \sigma_{31}'\right]^{R_0}_0  
     \\
     &~~~+ \int_0^{R_0} \left( Z\left(U''+ \frac 1r U'-\frac{1}{r^2} U\right) + \left(\frac{K_0Pm_0}{D(m_0)}-1\right) U \right) \sigma_{31}r\,dr  \nonumber 
     \\
     &= -ZR_0 \frac{\alpha/R_0 J_1'(\alpha )}{J_1(\alpha)}\sigma_{31}(R_0)+ZR_0 \sigma_{31}'(R_0) 
     \nonumber\\
     &=
     -Z R_0 \left(\frac{K_2}{K_0^2} - \left(\rho_{20}+\frac{\rho_{22}}{2}\right)\frac{\alpha/R_0}{K_0} \frac{J_1'(\alpha)}{J_1(\alpha)}+\left(\rho_{20}+\frac{\rho_{22}}{2}\right)\sigma_{11}''(R_0)\right) \nonumber \\
     &=-Z R_0 \left(\frac{K_2}{K_0^2} + \left(\rho_{20}+\frac{\rho_{22}}{2}\right)\left(\sigma_{11}''(R_0)-\frac{\alpha/R_0}{K_0} \frac{J_1'(\alpha)}{J_1(\alpha)}\right)\right) 
\end{align}
and the RHS of \eqref{eq:s31_pde} is given only in terms of known functions. By equating both sides we obtain the desired formula~\eqref{eq:K2_general} for $K_2$.
\end{proof}

{
\section{Numerical implementation and verification}\label{app:numerics}
The asymptotic traveling-wave expansion is constructed up to second order in the velocity parameter \(V\), and then verified by evaluating the residuals of the full nonlinear system. The \(O(V)\) term in the expansion is computed analytically. More precisely, the base radius \(R_0\) is determined from the mass constraint, the constant state is \(m_0=1/(\pi R_0^2)\), and the bifurcation point \(\hat K_0\) as the first positive solution of a transcendental equation. Once these quantities are fixed, the first-order modes are available in closed form in terms of Bessel functions. In particular, the stress mode \(\sigma_{11}(r)\) and its derivatives are evaluated analytically, and the density mode \(m_{11}(r)\) is then recovered algebraically from the linearized equations.

The \(O(V^2)\) term is computed by numerically evaluating the integral representations of the solutions to \(n=0\) and \(n=2\) Fourier modes. More precisely, the intermediate equations are solved by explicit integral formulas of Green-function-type, with regularity at \(r=0\) and the boundary condition at \(r=R_0\) built into the representation. The corresponding stress correction is then reconstructed from a Helmholtz-type radial equation using Bessel-function integral formulas, and the myosin correction is recovered algebraically. Finally, the $n=0$ component is shifted by a constant so that the \(O(V^2)\) mass constraint is satisfied. This integral approach is numerically more robust than directly solving the singular radial ODE system, especially near \(r=0\).

The truncated approximation is assembled from the analytic \(O(V)\) contribution and the numerically computed \(O(V^2)\) modes. Its accuracy is assessed by substituting the approximation into the full bulk equations and boundary conditions and measuring the resulting residuals. Table~\ref{tab:eoc_total_error} reports the total residual error for a sequence of values of \(V\), together with the experimental order of convergence (EOC). The observed EOC is approximately \(3\), which is consistent with the expected residual of \(O(V^3)\).

\begin{table}[ht]
\centering
\begin{tabular}{rrr}
\toprule
\(V\) & total error & EOC \\
\midrule
1.000000 & 2.016230 & --- \\
0.702422 & 0.686519 & 3.050072 \\
0.493396 & 0.235731 & 3.026270 \\
0.346572 & 0.081314 & 3.013343 \\
0.243440 & 0.028115 & 3.006675 \\
0.170998 & 0.009732 & 3.003313 \\
0.120112 & 0.003371 & 3.001633 \\
0.084370 & 0.001168 & 3.000794 \\
0.059263 & 0.000405 & 3.000364 \\
0.041628 & 0.000140 & 3.000079 \\
0.029240 & 0.000049 & 2.999582 \\
0.020539 & 0.000017 & 2.998707 \\
\bottomrule
\end{tabular}
\caption{Total residual error of the truncated second-order asymptotic approximation and the corresponding experimental order of convergence (EOC) as \(V\to0\). The observed third-order decay confirms that the residual is \(O(V^3)\).}\label{tab:eoc_total_error}
\end{table}
Finally, the coefficient \(K_2\) is verified independently through the third-order problem. Rather than solving the full \(O(V^3)\) system, we evaluate the corresponding solvability condition for the $n=1$ Fourier mode and check the compatibility of the third-order equations. This provides a second, independent numerical confirmation that the computed value of \(K_2\) is the one required for consistency of the asymptotic expansion.
}
\end{document}